\documentclass[reqno]{amsart}
\usepackage{amsmath, amsthm, amssymb, amstext}
\usepackage{mathrsfs}

\usepackage[left=2.5cm,right=2.5cm,top=2.5cm,bottom=2.5cm]{geometry}
\usepackage{hyperref,xcolor}
\hypersetup{pdfborder={0 0 0},colorlinks}

\usepackage{enumitem}
\allowdisplaybreaks
\usepackage{todonotes}
\usepackage{tikz}
\usetikzlibrary{decorations.fractals}
\usetikzlibrary{lindenmayersystems}
\usetikzlibrary[shadings]

\newtheorem{theorem}{Theorem}
\newtheorem{remark}[theorem]{Remark}
\newtheorem{lemma}[theorem]{Lemma}
\newtheorem{proposition}[theorem]{Proposition}
\newtheorem{corollary}[theorem]{Corollary}

\newtheorem{definition}[theorem]{Definition}
\newtheorem{example}[theorem]{Example}

\usepackage{graphicx}
\usepackage{float}

\def\neweq#1{\begin{equation}\label{#1}}
\def\endeq{\end{equation}}

\def\pp{\overrightarrow{p}(\cdot) }
\def\qq{\overrightarrow{q}(\cdot) }

\def\phi{\varphi}

\date{}

\def\pp{\overrightarrow{p}(\cdot) }
\def\qq{\overrightarrow{q}(\cdot) }

\def\fun(#1,#2,#3){\mathcal{E}_{_{#1}}(#2, #3)}
\def\sob(#1,#2){W^{1,#1}(#2)}
\def\forma(#1){\fun({\pp,\qq},\cdot,\cdot)}

\def\l2{L^2(\Omega)}

\def\funn(#1,#2,#3,#4,#5){\langle J_{#1} #2, #3 \rangle_{_{#4,#5}}}
\def\nn(#1,#2){\|#1\|_{_{#2}}}

\AtBeginDocument{{\noindent\small
{\em {\color{blue}Advances in Nonlinear Analysis} {\color{orange}(to appear)}}}
\vspace{9mm}}

\begin{document}

\title{Generalized quasi-linear fractional Wentzell problems}

\author{Efren Mesino\,-\,Espinosa,\,\,\,Alejandro V\'elez\,-\,Santiago}

\address{Efren Mesino\,-\,Espinosa\hfill\break
Department of Mathematical Sciences\\
University of Puerto Rico at Mayag\"uez\\
Mayag\"uez, Puerto Rico\, 00681}
\email{efren.mesino@upr.edu}

\address{Alejandro V\'elez\,-\,Santiago\hfill\break
Department of Mathematics\\
University of Puerto Rico at R\'io Piedras\\
San Juan, Puerto Rico\, 00925}
\email{alejandro.velez2@upr.edu,\,\,\,dr.velez.santiago@gmail.com}


\numberwithin{equation}{section}

\begin{abstract}
Given a bounded $(\epsilon,\delta)$-domain $\Omega\subseteq\mathbb{R\!}^N$ ($N\geq2$) whose boundary $\Gamma:=\partial\Omega$ is a $d$-set for $d\in(N-p,N)$, we investigate a generalized quasi-linear elliptic boundary value problem governed by the regional fractional $p$-Laplacian $(-\Delta)^s_{_{p,\Omega}}$ in $\Omega$, and generalized fractional Wentzell boundary conditions of type
$$C'_{p,s}\mathcal{N}^{p'(1-s)}u+\beta|u|^{q-2} u+\Theta^{\eta}_qu\,=\,g\indent\indent\indent\textrm{on}\,\,\Gamma,$$
where $\Theta^{\eta}_q$ stands as a nonlocal fractional-type $q$-operator on $\Gamma$ (also refered as a Besov $q$-map), $C'_{p,s}\mathcal{N}^{p'(1-s)}$ denotes the fractional $p$-normal derivative operator in $\Gamma$, and $p,\,q$ are two growth exponents acting on the interior and boundary, respectively (which are in general unrelated between each other). We first show that this model equation admits a unique weak solution, which is globally bounded in $\overline{\Omega}$. Furthermore, given two distinct weak solution related to this boundary value problem with different data values, we establish a priori $L^{\infty}$-estimates for the difference of weak solutions with upper bound depending in the differences of the respective interior and boundary data functions. Additionally, a Weak Comparison Principle is derived, and we conclude by establishing a sort of nonlinear Fredholm Alternative related to this generalized elliptic fractional model equation.
\end{abstract}
\maketitle

\section{Introduction}\label{sec1}

Let $\Omega\subseteq\mathbb{R\!}^N$ ($N\geq2$) be a bounded domain with ``sufficient geometry" over its boundary $\Gamma:=\partial\Omega$ (refer to Assumptions \eqref{A}(a,b)). We are concerned with a fractional-type quasi-linear elliptic equation involving the regional fractional $p$-Laplace operator $(-\Delta)^s_{_{p,\Omega}}$ in $\Omega$ and boundary conditions containing a generalized fractional $q$-Laplace-type operator on the boundary $\Gamma:=\partial\Omega$ (in the sense of Besov spaces and operators on the boundary). Boundary conditions of this kind will be called {\it Wentzell boundary conditions}, as a fractional counterpart of the classical Wentzell boundary value problems (also referred as Venttsel') introduced by Venttsel \cite{VENT60}. In the classical setting, $-\Delta_{_{q,\Gamma}}$ is known as the Laplace-Beltrami operator on $\Gamma$, which requires a certain level of regularity on the manifold $\Gamma$ in order to be well defined. For example, if $\Omega\subseteq\mathbb{R\!}^{\,2}$ is the classical Snowflake domain (see Example \ref{Ex1}(a)) and $1<q\leq 2$, then $-\Delta_{_{q,\Gamma}}$ is no longer well defined over $\Gamma$, but an abstract realization of a quasi-linear Wentzell-type boundary condition over the Snowflake domains has been introduced in \cite{LAN-VELEZ-VER18-1,LAN-VELEZ-VER15-1}. In the case of the fractional Wentzell-type operator, we are allowed to deal with less regularity structure on $\Gamma$, but there are almost no results obtained for models of fractional Wentzell type. It is also worth mentioning that the regional fractional Laplacian is not as known as the classical fractional Laplacian $(-\Delta)^s$, defined over all $\mathbb{R\!}^N$. A complete treatment of the regional fractional Laplace equation over non-smooth domains has been given by Warma \cite{warma2015fractional}.\\
\indent We now introduce the boundary value problem under consideration in this paper. Consider the following quasi-linear Wentzell-type problem, explicitly given by:
\begin{equation}\label{ep}
 \left\{
\begin{array}{ll}
(-\Delta)^s_{_{p,\Omega}}u+\alpha|u|^{p-2}u\,=\,f&\textrm{in}\,\,\Omega,\\ \\
C_{p,s}\mathcal{N}^{p'(1-s)}_{p}u+\beta|u|^{q-2}u+\Theta^{\eta}_qu\,=\,g&\textrm{on}\,\,\Gamma,\\
\end{array}
\right. \tag{$\mathcal{EP}$}
\end{equation}
where $C'_{p,s}\mathcal{N}^{p'(1-s)}$ denotes the fractional $p$-normal derivative operator in $\Gamma$. Then, under the conditions listed in Assumption \eqref{A}, we establish five main results which we briefly enumerate below:
\begin{enumerate}
\item The existence and uniqueness of a weak solution to problem \eqref{ep} (e.g. Theorem \ref{solvability}).
\item The global boundedness of weak solutions of problem \eqref{ep} under minimal and optimal conditions on the data (e.g. Theorem \ref{global-boundedness}(a)).
\item A $L^{\infty}$-estimate for the difference of two weak solution of equation \eqref{ep} with respect to the difference of the corresponding data for each weak solution (e.g. Theorem \ref{global-boundedness}(b)).
\item A Weak Comparison Principle for solution to the boundary value problem \eqref{ep} (e.g. Theorem \ref{WCP}).
\item A Nonlinear Fredholm Alternative associated with the equation \eqref{ep} (e.g. Theorem \ref{Fredholm}).
\end{enumerate}

In views of the above formulations, we address several important main observations concerning our equation of interest. Firstly, we observe that problem \eqref{ep} involves two exponents $p,\,q\in(1,\infty)$, which are allowed to be independent with respect to each other. Furthermore, the interior data in the problem depends solely on $p$ while the boundary data depends exclusively on $q$, and thus, the data values may be unrelated with respect to each other. As expected, these generalities bring much more complications in the proofs of most of the results. The mentioned independence between the interior power $p$ and the boundary power $q$ implies that problem (\ref{ep}) can be interpreted (in the variational setting) as a
system of equations of the form
\begin{equation}
\label{var-system}\left\{
\begin{array}{lcl}
C_{N,p,s}\displaystyle\int_{\Omega}\displaystyle\int_{\Omega} \frac{|u(x)-u(y)|^{p-2}(u(x)-u(y))(\varphi(x)-\varphi(y))}{|x-y|^{N+sp}}  \, \mathrm{d}x \, \mathrm{d}y + \displaystyle\int_{\Omega}\alpha |u|^{p-2}u\varphi  \, \mathrm{d}x=\displaystyle\int_{\Omega}f\varphi\,dx,\\
\indent\\
\displaystyle\int_{\Gamma} \displaystyle\int_{\Gamma} \frac{|v(x)-v(y)|^{q-2}(v(x)-v(y))(\psi(x)-\psi(y))}{|x-y|^{d+\eta q}} \, \mathrm{d}\mu_{x} \, \mathrm{d}\mu_{y} + \displaystyle\int_{\Gamma}\beta |v|^{q-2}v\psi \, \mathrm{d}\mu_{x}=\displaystyle\int_{\Gamma}g\psi\,\mathrm{d}\mu_{x},
\end{array}
\right.
\end{equation}\indent\\
for $v:=u|_{_{\Gamma}}$, $\psi:=\phi|_{_{\Gamma}}$. In our knowledge, this is first time that a problem of this type has been considered, with all these structures.\\
\indent For the classical local problem involving the $p$-Laplace operator on the interior and the $q$-Laplace-Beltrami operator on the boundary, Gal and Warma \cite{GAL-WAR} introduced this problem, and they established the results (1)-(3) and (5) stated above, assuming that $\min\{p,q\}\geq2$ and that $p,\,q$ have a sort of relationship between them. These two conditions were removed in \cite{diaz2022generalized}, where the problem was also generalized to anisotropic structures. Further properties of local parabolic problems under these generalized structures were recently given in \cite{CAR-HEN-AVS24}. In particular, all the conditions (1)-(5) above were established in \cite{diaz2022generalized}, and our problem \eqref{ep} constitutes a fractional counterpart of the boundary value problem and results in 
\cite{diaz2022generalized} (in the constant setting). Since the calculations and estimates over nonlocal-type integrals bring further complications and additional cases not found in the classical local case, it becomes substantial to provide complete justification to all the derivations and estimates.\\
\indent Fractional Wentzell-type boundary value problems are barely known. For the linear case ($p=2$), Gal and Warma \cite{gal2016elliptic} introduced these model equations, and a complete realization of a diffusion problem of Wentzell-type was recently developed in \cite{G-M-S-AVS25}. For the quasi-linear case with $q=p$, Creo and Lancia \cite{creo2021fractional} were the first authors to formulate problem \eqref{ep} under the same generality assumptions on the domain. In fact, they were the ones who extended the fractional integration by part formula to non-Lipschitz domains (see Theorem \ref{fgf}). On the other hand, they assume that the coefficients $\alpha,\,\beta$ are such that $\alpha\equiv0$ and $\beta\in L^{\infty}_{\mu}(\Gamma)$ is strictly positive. Under these conditions, Creo and Lancia \cite{creo2021fractional} established the global boundedness of weak solutions under the hypothesis $p\geq2$. Motivated by the model equation introduced by Creo and Lancia in \cite{creo2021fractional}, in our case we extend their model to the case when $p\neq q$ (possibly without any relation between $p$ and $q$), we allow unbounded interior and boundary local coefficients (see condition \eqref{A}(d)), and we establish our results including the singular cases $p,\,q\in(1,2)$. Particularly, we allow situations in which our interior operator is in singular mode ($1<p<2$) while the boundary part is non-singular ($q\geq2$), and the other way around. Henceforth, under the generality of our assumption, model equation \eqref{ep} is new, and the set of results (1)-(5) are established for this generalized fractional boundary value problem for the first time.\\
\indent Fractional Laplacian-type operators are very important in probability, as they are generators of $s$-stable processes (Lévy flights in some of the physical literature), widely used to model systems of stochastic dynamics with applications in operation research, mathematical finance, risk estimate, and generalized Fokker-Plank equations for nonlinear
stochastic differential equations driven by Levy noise, among many others (e.g. \cite{APPL04,S-L-D-Y-L} and the references therein). Models of nonlocal structure are tied to many applications, including continuum mechanics, graph theory, nonlocal wave equations, jump processes, image analysis, machine learning, kinetic equations, phase transitions,
nonlocal heat conduction, and the peridynamic model for mechanics (among others). As mentioned in \cite{warma2015fractional}, the fractional Laplacian and fractional derivative operators are commonly used to model anomalous diffusion. Physical phenomena exhibiting this property cannot be modeled accurately by the usual advection-dispersion equation; among others, we mention turbulent flows and chaotic dynamics of classical conservative systems. For more information, properties, and applications of nonlocal-type operators and structures, we refer to \cite{DU-GUNZ-LEH13,DU-GUNZ-LEH12,GU-TANG24,GUNZ-LEH10,warma2015fractional} (among many others).\\
\indent The organization of the paper is the following. Section \ref{sec2} reviews the fundamental concepts, definitions, and known results which will be applied in the subsequent sections. Some particular examples of non-smooth domains having the structure suitable for the paper are also given in this section. In section \ref{sec3}, we first state the main assumptions that will be carried out throughout the paper, and then we proceed to stablish the unique solvability of problem \eqref{ep}. Is section \ref{sec4}, we develop a priori estimates for weak solutions to problem \eqref{ep}, as well as for the difference of weak solutions. This is the crucial part of the paper, where the generality of the growth interior and boundary exponents produce many complications, and in particular, to derive $L^{\infty}$-bounds for weak solutions depending on the difference of their corresponding data, one needs to consider four cases related to the singularity/non-singularity of the interior/boundary part of the variational problem. However, having a priori estimates for the difference of weak solutions is a very useful result, which also provides continuous dependence between the data and the corresponding weak solutions, in the sense that the ``closeness" between the data terms imply the closeness of the corresponding weak solutions. As a direct consequence, global boundedness of weak solutions to problem \eqref{ep} is established in this section. Section \ref{sec5} features inverse positivity results for weak solutions of problem \eqref{ep}, and at the end, a Weak Comparison Principle is derived for problem \eqref{ep}. Finally, in section \ref{sec6}, a nonlinear Fredholm Alternative is given for the associated functional related to problem \eqref{ep}.

\section{Preliminaries}\label{sec2}

We state fundamental concepts, notations, and important tools for the development of the work in this paper. We introduce some function spaces along with their essential properties. We also define fractional Sobolev spaces and Besov spaces and provide some embedding theorems. Furthermore, we introduce the regional fractional $p$-Laplacian and the fractional normal derivative. Examples of irregular domains in which the framework of the paper applies are given, and at the end, we provide some analytical results, which will be of utility for the establishment of main results of the paper. 

\subsection[(epsilon, delta)-Domains and d-Sets]{\texorpdfstring{$(\epsilon,\delta)$-Domains and $d$-Sets}{(epsilon, delta)-Domains and d-Sets}}\label{subsec2.1}

Here, we will introduce two key concepts about sets that will be considered in the development of this research. For a deeper exploration of these concepts, readers are encouraged to refer to the following references: \cite{jones1981quasiconformal,jonsson1984function}.

\begin{definition} \label{Def1}
(see \cite{jones1981quasiconformal})   Let $\Omega \subset \mathbb{R}^{N}$ be open and connected, and let   $$d(x):=\inf\limits_{y \in \Omega^{c}}|x-y|, \quad x \in \Omega.$$ 
    We say that $\Omega$ is an $(\epsilon,\delta)$-\textbf{domain} if, whenever $x,y \in \Omega$ with $|x-y|<\delta$, there exists a rectifiable arc $\gamma \subset \Omega$ joining $x$ to $y$ such that 
\begin{equation}
    l(\gamma) \leq \frac{1}{\epsilon}|x-y| \hspace{0.3cm} \text{and} \hspace{0.3cm} d(z)\geq \frac{\epsilon|x-z||y-z|}{|x-y|} \hspace{0.3cm} \text{for every} \hspace{0.3cm} z \in \gamma.
\end{equation}
\end{definition}

\begin{definition}\label{dset}
    
(see \cite{jonsson1984function}) A closed non-empty set $F \subset \mathbb{R}^{N}$ is a \textbf{$d$-set} (for $0<d\leq N$) if there exist a Borel measure $\mu$ with supp $\mu = F$ and two positive constants $c_{1}$ and $c_{2}$ such that
\begin{equation}
    c_{1}r^{d}\leq \mu(B(x,r) \cap F)\leq c_{2}r^{d}, \hspace{0.3cm} \forall x \in F.
\end{equation}
The measure $\mu$ is called \textbf{$d$-Ahlfors measure}.
\end{definition}

\indent Examples of sets fulfilling both of the above definitions include many non-smooth and non-Lipschitz domains, as we will show later below.
\subsection{Fractional Sobolev spaces, Besov spaces, and other function spaces}\label{subsec2.2}

In this section, we will introduce the concepts of fractional Sobolev spaces and Besov spaces. Additionally, we will discuss some important properties of these spaces that are essential for the development of this work. For more information, refer to the following references: \cite{edmunds2022fractional,leoni2023first,di2012hitchhikers,warma2016fractional,creo2021fractional,jonsson1984function,jonsson1994besov}.\\
\indent Firstly, given $r\in[1,\infty]$ and a $\mu$-measurable set $E$ in $\mathbb{R\!}^N$, we denote by $L^r_{\mu}(E)$ the basic $L^r$-spaces with respect to the measure $\mu$, and for the particular case in which $\mu(\cdot)=|\cdot|$ is the $N$-dimensional Lebesgue measure on $E$, we write $L^r_{\mu}(E):=L^r(E)$. The norm for the $L^r$-spaces will be denoted by $\|\cdot\|_{_{r,E}}.$

\begin{definition}\label{Sobolev}
(see \cite{warma2016fractional,leoni2023first}) Let $\Omega \subseteq \mathbb{R}^{N}$ be an  open set. For $p \in [1, \infty)$ and $s \in (0,1)$. The \textbf{fractional Sobolev space} on $\Omega$ is defined as
$$ W^{s,p}(\Omega):=\left \{ u\in L^{p}(\Omega) \ | \ \int_{\Omega}\int_{\Omega}\frac{|u(x)-u(y)|^{p}}{|x-y|^{N+sp}}\, \mathrm{d}x \, \mathrm{d}y <\infty  \right \},$$
endowed with the norm
$$||u||_{_{W^{s,p}(\Omega)}}^{p}:=\|u\|_{_{p,\Omega}}^{p}+\int_{\Omega}\int_{\Omega}\frac{|u(x)-u(y)|^{p}}{|x-y|^{N+sp} }\, \mathrm{d}x \, \mathrm{d}y.$$
    
\end{definition}

\begin{theorem}
   (see \cite{leoni2023first}) The fractional Sobolev space $(W^{s,p}(\Omega),\|\cdot\|_{_{W^{s,p}(\Omega)}})$ is a Banach space. Moreover, $W^{s,p}(\Omega)$ is reflexive for $1 < p < \infty$.
\end{theorem}

\begin{definition}\label{Besov}
(see \cite{creo2021fractional,danielli2006non}) Let $F$ be a $d$-set with respect to a $d$-Ahlfors measure $\mu$. Given $q\in [1,\infty)$ and $\eta \in (0,1)$, the \textbf{Besov space} on $F$ relative to the measure $\mu$ is defined as
$$\mathbb{B}^{q}_{\eta}(F):= \left \{  u\in L^{q}_{\mu}(F) \ | \  \int_{F}\int_{F}\frac{|u(x)-u(y)|^{q}}{|x-y|^{d+\eta q}}\, \mathrm{d}\mu_{x} \, \mathrm{d}\mu_{y} < \infty \right \},$$
endowed with the norm
$$||u||_{_{\mathbb{B}_{\eta}^{q}(F)}}^{q}:=||u||_{_{q,F}}^{q}+\int_{F}\int_{F}\frac{|u(x)-u(y)|^{q}}{|x-y|^{d+\eta q}}\, \mathrm{d}\mu_{x} \, \mathrm{d}\mu_{y}. $$
\end{definition}

\begin{theorem}
   (see \cite{jonsson1984function}) The Besov space $(\mathbb{B}^{q}_{\eta}(F),||\cdot||_{_{\mathbb{B}_{\eta}^{q}(F)}})$ is a Banach space. Moreover, $\mathbb{B}^{q}_{\eta}(F)$ is reflexive for $1 < q < \infty$.
\end{theorem}

\indent For the domains under consideration in this paper, a characterization of the trace space of $W^{s,p}(\Omega)$ is given in the next known result.

\begin{theorem}
(see \cite{creo2021fractional,jonsson1984function})
    Let $\Omega \subseteq \mathbb{R}^{N}$ be a bounded $(\epsilon,\delta)$-domain with boundary $\partial \Omega$ a $d$-set, $\frac{N-d}{p}<s<1$ and $\eta=s-\frac{N-d}{p} \in (0,1)$. Then $\mathbb{B}_{\eta}^{p}(\partial \Omega)$ is the trace space of $W^{s,p}(\Omega)$ in the following sense:
    \begin{itemize}
        \item [$(i)$] $\gamma_{0}$ is a continuous linear operator from $W^{s,p}(\Omega)$ to $\mathbb{B}_{\eta}^{p}(\Gamma)$;

        \item [$(ii)$] there exists a  continuous linear operator Ext from $\mathbb{B}_{\eta}^{p}(\Gamma)$ to $W^{s,p}(\Omega)$ such that $\gamma_{0} \,  \circ \, Ext$ is the identity operator in $\mathbb{B}_{\eta}^{p}(\Gamma)$.
    \end{itemize}
    \end{theorem}

\begin{definition}
   (see \cite{leoni2023first}) Let $1\leq p < \infty$ and $0<s<1$. An open set 
$\Omega\subseteq\mathbb{R}^N$ is called a $W^{s,p}$-\textbf{extension domain}, if there exists a continuous linear operator
$\mathcal{E}:W^{s,p}(\Omega)\rightarrow W^{s,p}(\mathbb{R}^N)$ such that $\mathcal{E}u=u$  a.e. on $\Omega$.
\end{definition}

It is known that if $\Omega\subseteq\mathbb{R\!}^N$ is a $N$-set, then $\Omega$ is a $W^{s,p}$-extension domain for $p\in[1,\infty)$ (e.g. \cite{jonsson1984function}). Now, we provide an important result about $(\epsilon,\delta)$-domains having as boundary a $d$-set.

\begin{theorem}(See \cite[Theorem 1.6]{creo2021fractional})
Let $1\leq p < \infty$, $0<s<1$, and $\Omega \subseteq \mathbb{R}^{N}$ be an $(\epsilon,\delta)$-domain with boundary a $d$-set. Then, $\Omega$ is a $W^{s,p}$-extension domain.
\end{theorem}

\indent In some instances, for $r,l \in [1,\infty)$, or $r=l=\infty$, we will deal with the space
$$\mathbb{X}^{r,l}(\Omega,\partial \Omega):=\{ (f,g) \ | \ f \in L^{r}(\Omega), \ g \in L_{\mu}^{l}(\partial \Omega) \},$$
endowed with the norm
$$|\|(f,g)\||_{_{\mathbb{X}^{r,l}(\Omega,\partial \Omega)}}=|\|(f,g)\||_{r,l}:=||f||_{r,\Omega}+||g||_{l,\partial \Omega},$$
and, if $r=l=\infty$,
$$|\|(f,g)\||_{_{\mathbb{X}^{\infty,\infty}}}=|\|(f,g)|\|_{_{\infty}}:=\max \{||f||_{\infty,\Omega},||g||_{\infty,\partial \Omega} \}.$$

\subsection{Embedding Theorems}\label{subsec2.3}

The following embedding results are important and will be highly applied throughout the rest of the paper. We will present results regarding fractional Sobolev spaces and Besov spaces.

First, we set 
\begin{equation}\label{critical-number}
    p^{*}:=\frac{Np}{N-sp}, \quad  (p^{*})':=\frac{Np}{Np-N+sp}, \quad    q^{*}:=\frac{dq}{d-\eta q},   \quad (q^{*})':=\frac{dq}{dq-d+\eta q}.
\end{equation}

\indent We first state the interior embedding result.

\begin{theorem} \label{sobolev-embedding}
(see \cite[Theorem 6.7]{di2012hitchhikers}). 
 Let $s \in (0,1)$ and $p \geq 1$ be such that $sp<N$. Let $\Omega \subseteq \mathbb{R}^{N}$ be a bounded $W^{s,p}$-extension domain. Then there exists a positive constant $c_{1}=c_{1}(N,s,p,\Omega)$ such that, for every $u\in W^{s,p}(\Omega)$, we have 
\begin{equation}
    ||u||_{_{r_{1},\Omega}}\leq c_{1} ||u||_{_{W^{s,p}(\Omega)}},
\end{equation}
that is, the space $W^{s,p}(\Omega)$ is continuously embedded in $L^{r_{1}}(\Omega)$  for any $r_{1} \in [1,p^{*}]$. Moreover, if in addition $r_1<p^{*}$, then the embedding is compact.
 \end{theorem}

 \indent For the next result, we use the first part of \cite[Theorem 11.1]{danielli2006non} together with the extension property and the fact that $C^{\infty}(\overline{\Omega})$ is dense in the fractional Sobolev spaces (following the idea as in \cite[chapter 10]{danielli2006non}; see also \cite[derivation of inequality (2.3)]{CHEN-KUM03} for the linear case) to obtain the following boundary embedding. This results justifies the fact that one may be able to consider independent growth exponents in the interior and at the boundary.

\begin{theorem}\label{besov-embedding}(see \cite{danielli2006non})
Let $\eta \in (0,1)$ and $q \geq 1$ be such that $N-\eta q<d<N$. Let $F$ be a $d$-set with respect to a $d$-Ahlfors measure $\mu$. Then $\mathbb{B}^{q}_{\eta}(F)$ is continuously embedded in $L^{r_{2}}(F)$ for every $r_{2} \in [1,q^{*}]$, i.e. there exists a positive constant $c_{2}$ such that
\begin{equation*}
    ||u||_{_{r_{2},F}} \leq c_{2}||u||_{_{\mathbb{B}^{q}_{\eta}(F)}}, \hspace{0.5cm} \forall u \in \mathbb{B}^{q}_{\eta}(F).
\end{equation*}
\end{theorem}

\begin{remark}\label{consequences}
Let $\Omega\subseteq\mathbb{R\!}^N$ be a bounded $(\epsilon,\delta)$-domain whose boundary is a $d$-set ($N\geq2$). Then the following hold.
\begin{enumerate}
\item[(a)]\,\, From \cite{SHV07}, one clearly deduce that $\Omega$ is a $N$-set.
\item[(b)]\,\, Since the maps  $W^{s,p}(\Omega)\hookrightarrow L^{r_p}(\Omega)$ and $\mathbb{B}^{q}_{\eta}(\Gamma)\hookrightarrow L^{r_q}_{\mu}(\Gamma)$ are compact whenever $r_p \in [1,p^{*})$ and $r_q \in [1,q^{*})$, in views of \cite[Lemma 2.4.7]{NITTKA2010}, for any $\epsilon>0$, there exists constants $C_{\epsilon},\,C'_{\epsilon}>0 $ such that
\begin{equation}\label{epsilon-interior-trace}
\|v\|^p_{_{r_p,\Omega}}\,\leq\,\epsilon\int_{\Omega}\int_{\Omega}\frac{|u(x)-u(y)|^{p}}{|x-y|^{N+sp}}\, \mathrm{d}x \, \mathrm{d}y+C_{\epsilon}\|v\|^p_{_{p,\Omega}}
\end{equation}
and
\begin{equation}\label{epsilon-boundary}
\|w\|^q_{_{r_q,\Gamma}}\,\leq\,\epsilon\int_{\Gamma}\int_{\Gamma}\frac{|u(x)-u(y)|^{q}}{|x-y|^{d+\eta q}}\, \mathrm{d}\mu_{x} \, \mathrm{d}\mu_{y}+C'_{\epsilon}\|w\|^q_{_{q,\Gamma}},
\end{equation}
for each $v\in W^{s,p}(\Omega)$ and $w\in \mathbb{B}^{q}_{\eta}(\Gamma)$.
\item[(c)]\,\, If $s\leq (N-d)/p<1$, then by \cite[Corollary 4.9]{warma2015fractional}, we obtain that $W^{s,p}(\Omega)=W^{s,p}_0(\Omega):=\overline{C^{\infty}_c(\Omega)}^{^{\|\cdot\|_{_{W^{s,p}(\Omega)}}}}$.
\end{enumerate}
\end{remark}

\subsection[The Regional Fractional p-Laplacian]{\texorpdfstring{The Regional Fractional $p$-Laplacian}{The Regional Fractional p-Laplacian}}\label{subsec2.4}

In this section, we focus on the regional fractional $p$-Laplacian, which has been previously studied in the linear case ($p=2$) in \cite{di2012hitchhikers, leoni2023first, guan2006integration}. However, given our interest in a more general approach, we concentrate on the details provided in \cite{warma2016fractional}. Additionally, we introduce the $p$-fractional normal derivative via a $p$-fractional Green formula given in \cite{creo2021fractional}.

\begin{definition}(see \cite{warma2016fractional})
    Let $s \in (0,1)$, $p>1$ and $\Omega \subseteq \mathbb{R}^{N}$, we define the space 
    \begin{equation*}
        \mathcal{L}^{p-1}_{s}(\Omega):= \left \{  u: \Omega \rightarrow \mathbb{R} \ \text{measurable} \ | \ \int_{\Omega} \frac{|u(x)|^{p-1}}{(1+|x|)^{N+sp}} \, \mathrm{d}x < \infty 
\right \}.
    \end{equation*}
\end{definition}

\begin{definition}(see \cite{warma2016fractional})
Let $s \in (0,1)$, $p>1$ and $\Omega \subset \mathbb{R}^{N}$ be an open set. For $u \in \mathcal{L}^{p-1}_{s}(\Omega) $ and $x \in \Omega$, we define the \textbf{regional fractional $p$-Laplacian} $(-\Delta)^{s}_{_{p,\Omega}}$ as follows 
\begin{eqnarray}
    (-\Delta)^s_{_{p,\Omega}}u(x)&:=&C_{N,p,s}\, \textup{P.V.}\int_{\Omega}|u(x)-u(y)|^{p-2}\frac{u(x)-u(y)}{|x-y|^{N+sp}}\,\mathrm{d}y  \nonumber \\ 
    &=&C_{N,p,s}\, \lim_{\varepsilon \rightarrow 0^{+}}\int_{\{y \in \Omega \, : \, |x-y|> \varepsilon\}}|u(x)-u(y)|^{p-2}\frac{u(x)-u(y)}{|x-y|^{N+sp}}\,\mathrm{d}y,
\end{eqnarray}
provided that the limit exists. Where the positive normalized constant $C_{N,p,s}$ is given by 
\begin{equation*}
    C_{N,p,s}:= \frac{s4^{s}\Gamma((ps+p+N-2)/2)}{\pi^{N/2}\Gamma(1-s)},
\end{equation*}
and $\Gamma(\cdot)$ is the usual Gamma function.
\end{definition}

Now, we introduce the notion of $p$-fractional normal derivative for the domains under study in this work, and a corresponding integration by parts formula. For more details and proof, refer to \cite[Theorem 2.2]{creo2021fractional}. For this, we define the space
\begin{equation*}
    V((-\Delta)^s_{_{p,\Omega}},\Omega):=\{ u \in W^{s,p}(\Omega) \ | \ (-\Delta)^s_{_{p,\Omega}}u \in L^{p'}(\Omega) \ \text{in the sense of distributions}\}.
\end{equation*}

\begin{theorem} (Fractional Green Formula; see \cite[Theorem 2.2]{creo2021fractional})\label{fgf}
    Let  $\Omega \subseteq \mathbb{R}^{N}$ be a  bounded  $(\epsilon,\delta)$-domain, which can be approximated  by a sequence 
$\{\Omega_{n}\}$ of domains such that for every $n \in \mathbb{N}$:
\begin{equation*}
    (\mathcal{H})\left \{ \begin{array}{l}
         \Omega_{n} \ \text{is bounded and Lipschitz;} \\
         \Omega_{n}\subseteq \Omega_{n+1}; \\
\Omega=\bigcup\limits_{n=1}^{\infty}\Omega_{n},
    \end{array}\right.
    \end{equation*}
and its boundary $\partial\Omega$ is a  $d$-set, for $p>1$, $\frac{N-d}{p}<s<1$ and $\eta=s-\frac{N-d}{p}$. Then, there exists a bounded linear operator $\mathcal{N}_{p}^{p'(1-s)}$ from $V((-\Delta)^s_{_{p,\Omega}},\Omega)$ to $\mathbb{B}^{p}_{\eta}(\partial \Omega)^{\ast}$, the following generalized Green formula holds for every  $u \in V((-\Delta)^s_{_{p,\Omega}},\Omega)$ and $v \in W^{s,p}(\Omega)$:\\ 
$\left \langle C_{p,s} \mathcal{N}_{p}^{p'(1-s)}u, v|_{\partial \Omega} \right \rangle_{\mathbb{B}^{p}_{\eta}(\partial \Omega)^{\ast},\mathbb{B}^{p}_{\eta}(\partial \Omega)}$
\begin{equation*}
    =-\int_{\Omega} (-\Delta)^s_{_{p,\Omega}}u \, v \, \mathrm{d}x + \frac{C_{N,p,s}}{2}\int_{\Omega} \int_{\Omega} \frac{|u(x)-u(y)|^{p-2}(u(x)-u(y))(v(x)-v(y))}{|x-y|^{N+sp}}  \, \mathrm{d}x \, \mathrm{d}y,
\end{equation*}
where
$$C_{p,s}:=\frac{(p-1)C_{1,p,s}}{(sp-(p-2))(sp-(p-2)-1)}\displaystyle\int^{\infty}_0\frac{|t-1|^{p-1-sp}-(t\vee1)^{p-sp-1}}{t^{p-sp}}\,\mathrm{d}t.$$
\end{theorem}

\begin{remark}
   Note that if \( p'(1-s) = 2 - \gamma \), where \(\gamma = \frac{ps - 1}{p - 1} + 1\), then we have the usual notation for the \( p \)-fractional normal derivative.
\end{remark}

\subsection{Examples of domains}\label{subsec2.5}

We now collect some examples of non-smooth regions exhibiting the required conditions for the validity of Theorem \ref{fgf}.

\begin{example}\label{Ex1}
	We present three concrete examples of non-Lipschitz domains fulfilling all the contitions mentioned in Theorem \ref{fgf}.
	\begin{enumerate}
		\item[\textnormal{(a)}]
			The most known example of a locally uniform domain with fractal boundary (and thus failing to be a Lipschitz domain) is the classical Koch snowflake domain in $\mathbb{R}^{2}$, see Figure \ref{KochSnowflake}. In this case, $\mu(\cdot)=\mathcal{H}^d(\cdot)$ is the $d$-dimensional Hausdorff measure supported on $\Gamma$ for $d:=\log(4)/\log(3)$.
			\begin{figure}[H]
				\begin{center}
					\includegraphics[scale=1.5]{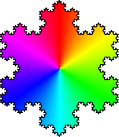}
					\caption{The 2-dimensional Koch snowflake domain.}
					\label{KochSnowflake}
				\end{center}
			\end{figure}
            The first result concerning traces of Sobolev functions over this particular set were introduced by Wallin \cite{wallin1991trace}.
		\item[\textnormal{(b)}]
			Recently, Ferrer and V\'{e}lez-Santiago \cite{Ferrer-Velez-Santiago-2023} constructed a class of Koch-type domains $\{\Omega_m\}$ in $\mathbb{R}^{3}$ for a given $m>2$ being an integer not having $3$ as a factor. Such domains are called Koch-type crystals, see Figure \ref{fig:2}. Here  $\mu(\cdot)=\mathcal{H}^d(\cdot)$ for $d:=\log(m^2+2)/\log(m)$.
			\begin{figure}[H]
				\includegraphics[scale=0.25]{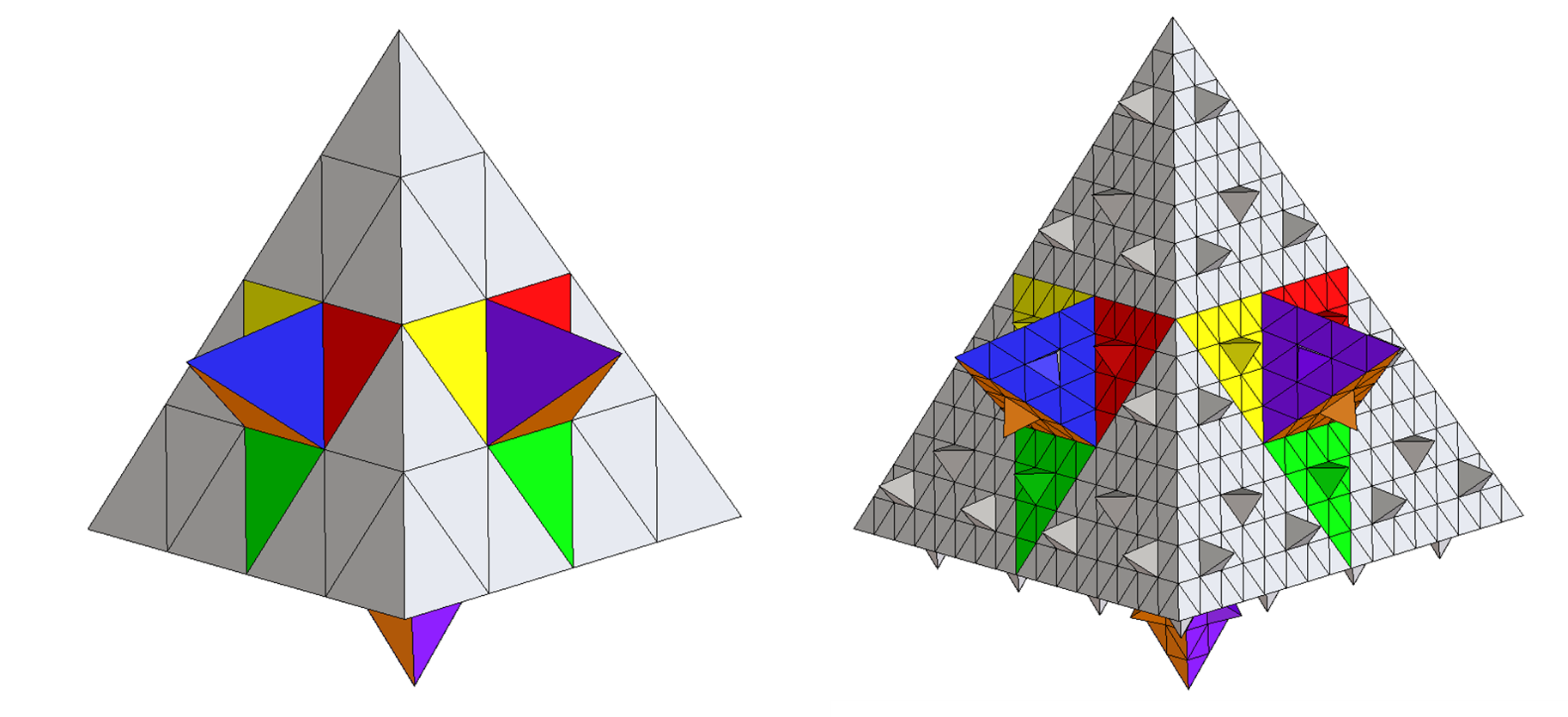}
				\caption{Koch 4-Crystal $\Omega_4$ (pre-Fractal structure: First and Second Iterations).}\label{fig:2}
			\end{figure}
		\item[\textnormal{(c)}]
			We consider ramified domains over $\mathbb{R}^{2}$ introduced by Achdou, Sabot, and Tchou \cite{Achdou-Sabot-Tchou-2006} and Achdou and Tchou \cite{ACH08}, see Figure \ref{T-shaped-trees}. In \cite{Achdou-Sabot-Tchou-2006,ACH08} it is shown that these ramified domains are $N$-sets, but are not $W^{1,2}$-extension domains, and thus from Jones \cite[Theorem 3]{jones1981quasiconformal}, it follows that these sets are not $(\epsilon,\delta)$-domains. However, since these domains are $2$-sets, it follows that these domains are $W^{s,p}$-extension domains for any $p\in[1,\infty)$ and $s\in(0,1)$, which can be approximated by a sequence of bounded Lipschitz domains fulfilling the condition ($\mathcal{H}$) in Theorem \ref{fgf}. Furthermore, by adapting the extension results in \cite{ACH08,jonsson1984function}, we can deduce the existence of a parameter $\eta^{\ast}$ such that the extension operator from $\mathbb{B}^{p}_{\eta^{\ast}}(\Gamma^{\infty})$ into $W^{s,p}(\Omega)$ is bounded, for $\Gamma^{\infty}$ denoting the ramified boundaries of the sets in Figure \ref{T-shaped-trees}. By taking functions in the set $\mathcal{W}_s(\Omega):=\left\{u\in W^{s,p}(\Omega)\mid u|_{_{\Gamma\setminus\Gamma^{\infty}}}=0\right\}$, one can deduce the validity of Theorem \ref{fgf} for this case. Thus, the results of this paper can be extended to sets like these ramified tree-type regions. These sets can be regarded as idealizations of the bronchial trees in a lungs system (e.g.\,\cite{Achdou-Sabot-Tchou-2006,ACH08}), which provides an important reason to investigate boundary value problems over such regions. For a more complete understanding of this domain and its importance with respect to a suitable diffusion problem, we refer to the recent manuscript by Silva-P\'erez and V{\'e}lez-Santiago \cite{Silva-Perez-Velez-Santiago-2024}.
			\begin{figure}[H]
				\begin{center}
					\includegraphics[scale=0.17]{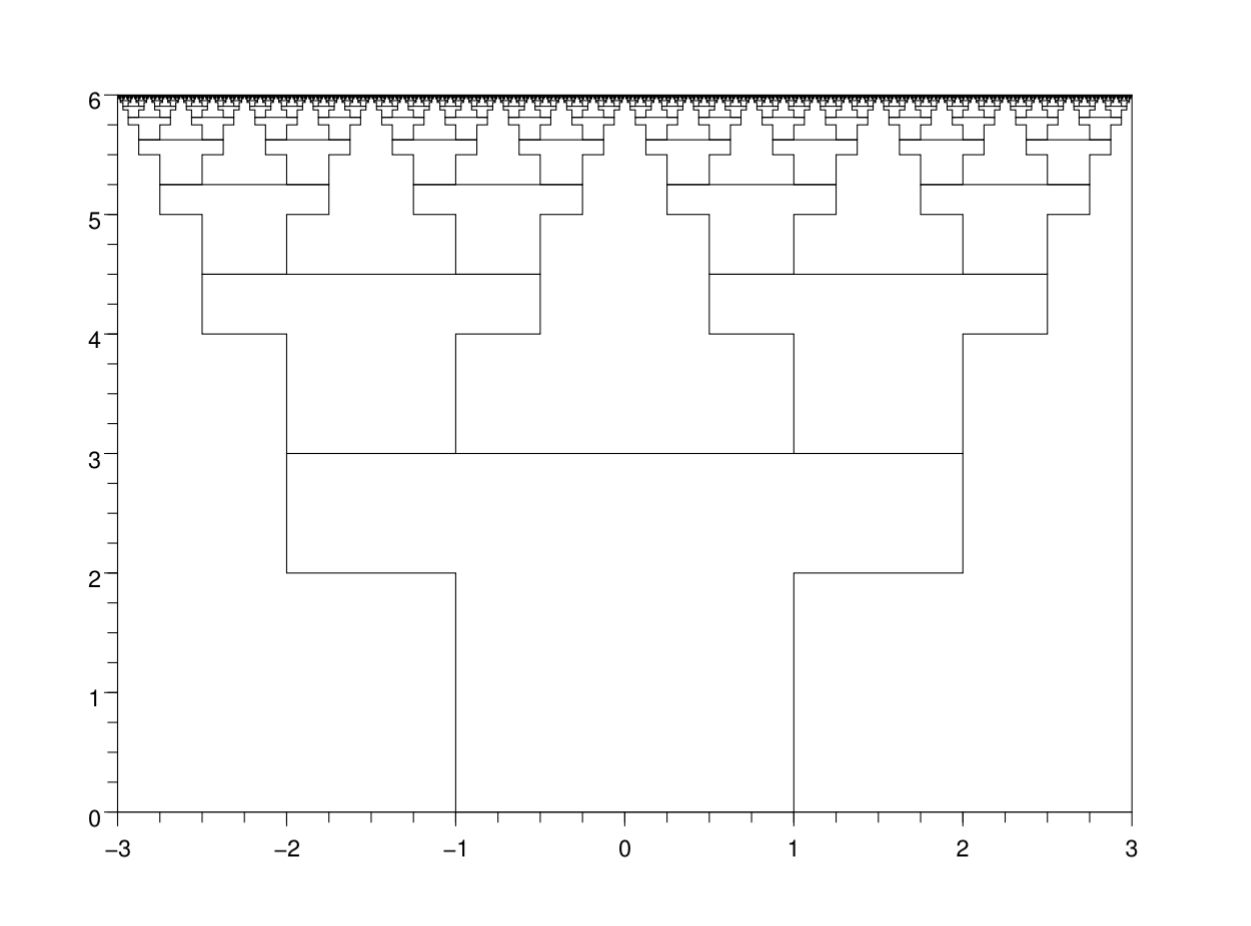}
					\includegraphics[scale=0.18]{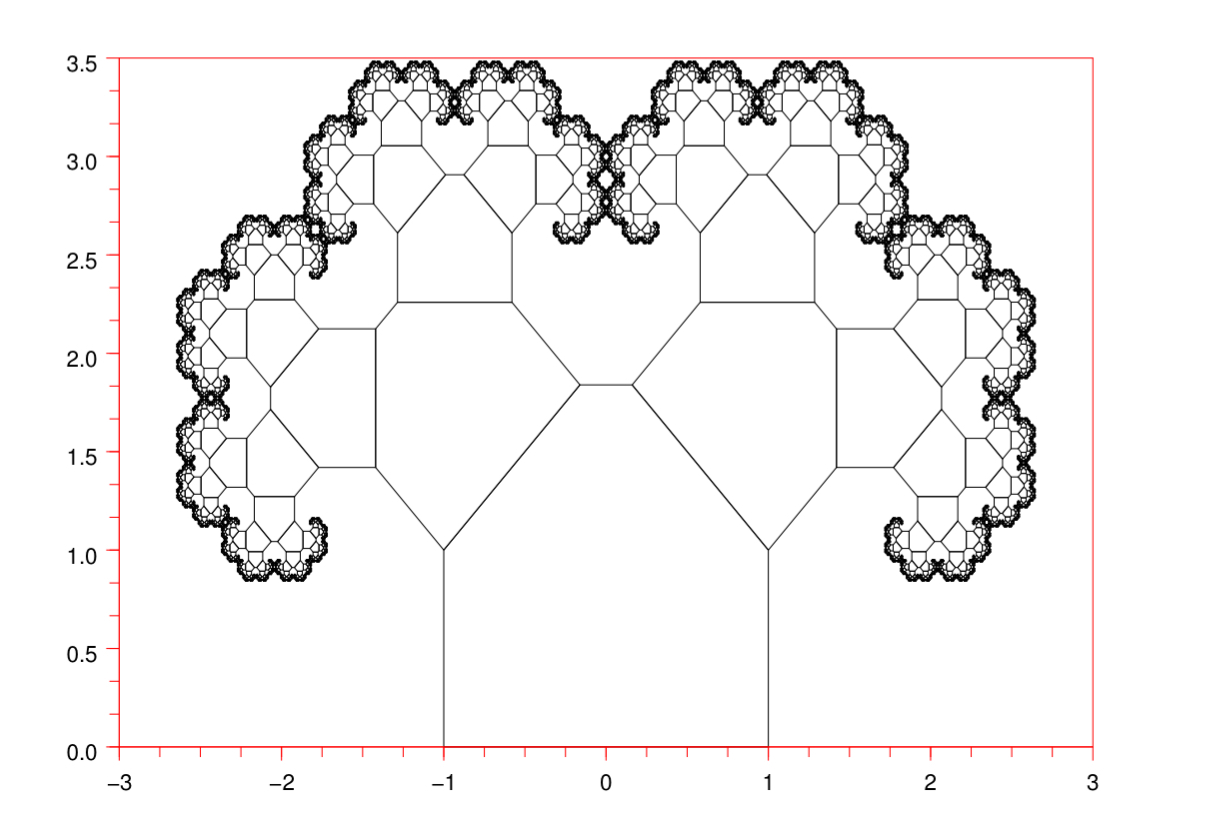}
					\caption{A T-shaped domain (see \cite{Achdou-Sabot-Tchou-2006}) on the left-hand side and a domain with ramified fractal boundary in the critical case $a=a^{\ast}\simeq 0.593465$ (see \cite{ACH08}) on the right-hand side.} \label{T-shaped-trees}
				\end{center}
			\end{figure}
		\end{enumerate}
\end{example}

\subsection{Some Analytical Tools}\label{subsec2.6}

We close this section by presenting some analytical results which will be very helpful in the subsequent parts of the paper.

\begin{lemma}\label{est}(see \cite{diaz2022generalized}). Let $\varphi= \varphi(t)$ be a non-negative, non-increasing function on a half line $\{  t\geq k_{0} \geq 0 \}$, such that there exist constants $c,\alpha_{1},\alpha_{2} > 0$, and there exists $\delta >1$ with 
\begin{equation*}
    \varphi(h) \leq c [(h-k)^{-\alpha_{1}}+(h-k)^{-\alpha_{2}}]\varphi(k)^{\delta},
\end{equation*}
 for $h>k\geq k_{0}$. Then 
\begin{equation*}
\varphi(k_{0} + \varsigma_{1} + \varsigma_{2})=0,
\end{equation*}
where 
\begin{equation*}
\varsigma_{1}^{\alpha_{1}}=\varsigma_{2}^{\alpha_{2}}=c \ \varphi(k_{0})^{\delta - 1}2^{(\alpha_{1}+\alpha_{2})\delta(\delta-1)}.
\end{equation*}
\end{lemma}

\begin{proposition}\label{duality}(see \cite{biegert2010priori}).
 Let $a,b \in \mathbb{R\!}\,$ and $r\in (1, \infty)$. Then there exists a constant $c_{r}>0$ such that
\begin{equation} \label{r>1}
    \left (  |a|^{r-2}a-|b|^{r-2}b\right )(a-b)\geq  c_{r} \left ( |a|+|b| \right )^{r-2}|a-b|^{2}\geq 0.
\end{equation}
If in addition $r \in [2,\infty)$, then there exists a constant $c^{*}_{r} \in (0,1]$ such that 
\begin{equation}\label{r>2}
    \left (  |a|^{r-2}a-|b|^{r-2}b\right )(a-b)\geq  c^{*}_{r} |a-b|^{r},
\end{equation}
and also in this case there is a constant $c'_{r} \in (0,1]$ such that
\begin{equation}\label{r>2*}
    \textrm{sgn}(a-b)\left (  |a|^{r-2}a-|b|^{r-2}b\right )\geq  c'_{r}  |a-b|^{r-1}.
\end{equation}
For $r\in (1,2]$ and $a \neq b$, we have 
\begin{equation}
    \left <  |a|^{r-2}a-|b|^{r-2}b,a-b\right> 
  \left [  |a|^{r}+|b|^{r}  \right]^{\frac{2-r}{r}}  \geq  (r-1)  |a-b|^{r},
\end{equation}
Finally, if $r \in (1,2]$ and $\epsilon > 0 $, then for each $a,b \in \mathbb{R}^{N}$ with $|a-b|\geq \epsilon \min \{ |a|,|b|  \}$ we find a constant $c_{r,\epsilon}$ such that
  \begin{equation}
      \left <  |a|^{r-2}a-|b|^{r-2}b,a-b\right> \geq c_{r,\epsilon}|a-b|^{r},
  \end{equation}
and 
  \begin{equation}
     \textrm{sgn}(a-b)\left (  |a|^{r-2}a-|b|^{r-2}b\right )\geq  c_{r,\epsilon}  |a-b|^{r-1}. 
  \end{equation}
\end{proposition}

\begin{proposition}\label{xita} 
(see \cite{biegert2010priori}). Let $\xi, \varsigma, c, \tau, \varrho \in [0,\infty)$ and $r\in [1, \infty)$, and assume that $\xi \leq \tau \epsilon^{-\varrho}\varsigma+\epsilon^{r}c$ for all $\epsilon > 0$. Then one has
\begin{equation}
    \xi \leq (\tau +1)\left [  \varsigma ^{\frac{r}{r+\varrho}}c^{\frac{\varrho}{r+\varrho}} + \varsigma \right ].
\end{equation}
\end{proposition}

\section{Solvability}\label{sec3}

\indent The aim of this section is to establish existence and uniqueness for weak solutions of the boundary value problem \eqref{ep}. Let us begin by stating below the main set of assumptions on problem \eqref{ep} which will be carried out through the rest of the paper. \\

\begin{enumerate}[label=\textnormal{(A)},ref=\textnormal{A}]
	\item\label{A}
		\begin{enumerate}
			\item[\textnormal{(a)}]
				$\Omega\subseteq\mathbb{R\!}^N$ is a bounded $(\epsilon,\delta)$-domain ($N\geq2$) which can be approximated by a sequence $\{\Omega_n\}_{_{n\in\mathbb{N\!}}}$ of domains satisfying, for each $n\in\mathbb{N\!}\,$:
                \begin{equation}\label{structure-domain}
                \left \{ \begin{array}{l}
         \Omega_{n} \ \text{is bounded and Lipschitz;} \\
         \Omega_{n}\subseteq \Omega_{n+1}; \\
\Omega=\bigcup\limits_{n=1}^{\infty}\Omega_{n}. 
    \end{array}\right.
    \end{equation}
			\item[\textnormal{(b)}]
				The boundary $\Gamma:=\partial\Omega$ is a $d$-set  with respect to a  measure $\mu$ supported on $\Gamma$, for $N-p<d<N$;\\
            \item[\textnormal{(c)}] $1<p<N$ and $1<q<d$ (unrelated);\\
			\item[\textnormal{(c)}]
				$s,\,\eta\in(0,1)$ (unrelated) with $s>(N-d)/p$;\\
			\item[\textnormal{(d)}]
				$\alpha\in L^{r_1}(\Omega)$ for $r_1>N/(sp)$ and $\beta\in L^{r_2}_{\mu}(\Gamma)$ for $r_2>d/(\tau q)$, with $\textrm{ess}\displaystyle\inf\limits_{x \in \Omega} \alpha(x) \geq \alpha_{0}$ and $\textrm{ess}\displaystyle\inf\limits_{x \in \Gamma} \beta(x) \geq \beta_{0}$ for some constants $\alpha_{0}, \beta_{0}>0$;\\
            \item[\textnormal{(e)}]
				$(f,g) \in \mathbb{X}^{r,l}(\Omega,\Gamma)$ for $r,\,l\in[1,\infty]$ given.
		\end{enumerate}
\end{enumerate}
\indent\\
\indent Next, define the following nonlinear forms:
$$\begin{array}{l}
\mathcal{E}_{_\Omega}(u,v):=C_{N,p,s}\displaystyle\int_{\Omega}\displaystyle\int_{\Omega} \frac{|u(x)-u(y)|^{p-2}(u(x)-u(y))(v(x)-v(y))}{|x-y|^{N+sp}}  \, \mathrm{d}x \, \mathrm{d}y + \displaystyle\int_{\Omega}\alpha |u|^{p-2}uv  \, \mathrm{d}x, 
\end{array}$$

$$\begin{array}{l}
      \mathcal{E}_{_\Gamma}(u,v):=\displaystyle\int_{\Gamma} \displaystyle\int_{\Gamma} \frac{|u(x)-u(y)|^{q-2}(u(x)-u(y))(v(x)-v(y))}{|x-y|^{d+\eta q}} \, \mathrm{d}\mu_{x} \, \mathrm{d}\mu_{y} + \displaystyle\int_{\Gamma}\beta |u|^{q-2}uv \, \mathrm{d}\mu_{x},
\end{array}$$
and 
\begin{equation}\label{form}
    \mathcal{E}_{p,q}(u,v):=\mathcal{E}_{_{\Omega}}(u,v) + \mathcal{E}_{_{\Gamma}}(u,v), \qquad \forall u,v \in \mathbb{W}_{p,q}(\Omega,\Gamma).
\end{equation}
Then, we have the following result for the above form.

\begin{lemma} \label{form-prop}
    The (nonlinear) form given by \eqref{form} satisfies $\mathcal{E}_{p,q}(u,\cdot) \in \mathbb{W}_{p,q}(\Omega,\Gamma)^{\ast}$ for each $u \in \mathbb{W}_{p,q}(\Omega,\Gamma)$. Moreover, $\mathcal{E}_{p,q}(\cdot ,\cdot)$ is hemicontinuous, strictly monotone, and coercive.
\end{lemma}

\begin{proof}
Fix $u \in \mathbb{W}_{p,q}(\Omega,\Gamma)$. Then in views of Theorems \ref{sobolev-embedding} and \ref{besov-embedding} together with H\"older’s inequality, we get the following calculations over each of the integral terms of the form $\mathcal{E}_{p,q}(\cdot ,\cdot)$: \\

 $\left | \displaystyle\int_{\Omega} \int_{\Omega} \frac{|u(x)-u(y)|^{p-2}(u(x)-u(y))(v(x)-v(y))}{|x-y|^{N+sp}}\, \mathrm{d}x \, \mathrm{d}y \right | $
\begin{eqnarray}\label{1.52}
    &\leq & \int_{\Omega} \int_{\Omega} \frac{|u(x)-u(y)|^{p-1}}{|x-y|^{\frac{p-1}{p}(N+ps)}}\frac{|v(x)-v(y)|}{|x-y|^{(N+ps)\frac{1}{p}}}\, \mathrm{d}x \, \mathrm{d}y \nonumber \\ \nonumber  \\
    &\leq &\left  [ \int_{\Omega}\int_{\Omega} \frac{|u(x)-u(y)|^{p}}{|x-y|^{N+sp}}\, \mathrm{d}x \, \mathrm{d}y  \right ]^{\frac{p-1}{p}} \left  [ \int_{\Omega}\int_{\Omega}\frac{|v(x)-v(y)|^{p}}{|x-y|^{N+sp}} \, \mathrm{d}x \, \mathrm{d}y  \right ]^{\frac{1}{p}} \nonumber \\ \nonumber \\
    &\leq & \| u \|_{_{W^{s,p}(\Omega)}}^{p-1}\| v \|_{_{W^{s,p}(\Omega)}} \leq |\|u\||_{_{\mathbb{W}_{p,q}(\Omega,\Gamma)}}^{p-1}|\|v\||_{_{\mathbb{W}_{p,q}(\Omega,\Gamma)}} 
\end{eqnarray}
In the same way, we deduce that\\[2ex]
$\left | \displaystyle\int_{\Gamma} \int_{\Gamma} \frac{|u(x)-u(y)|^{q-2}(u(x)-u(y))(v(x)-v(y))}{|x-y|^{d+\eta q}}\, \mathrm{d}\mu_{x} \, \mathrm{d}\mu_{y}\right | $
  \begin{eqnarray}\label{1.53}
         &\leq & \| u \|_{_{\mathbb{B}^{q}_{\eta}(\Gamma)}}^{q-1}\| v \|_{_{\mathbb{B}^{q}_{\eta}(\Gamma)}} \leq  |\|u\||_{_{\mathbb{W}_{p,q}(\Omega,\Gamma)}}^{q-1}|\|v\||_{_{\mathbb{W}_{p,q}(\Omega,\Gamma)}}
    \end{eqnarray} 
For the local terms, first we apply Theorem \ref{sobolev-embedding} to infer that\\[2ex]
$\left | \displaystyle\int_{\Omega} \alpha |u|^{p-2}uv\, \mathrm{d}x \right | \leq \| \alpha \|_{_{\frac{N}{sp},\Omega}} \| u \|_{_{\frac{Np}{N-sp},\Omega}}^{p-1} \| v \|_{_{\frac{Np}{N-sp},\Omega}}$\\
\begin{eqnarray}\label{1.54}
      & \leq &  C_{1}\| \alpha \|_{_{\frac{N}{sp},\Omega}}\| u \|_{_{W^{s,p}(\Omega)}}^{p-1}\| v \|_{_{W^{s,p}(\Omega)}}\leq C_{1}\| \alpha \|_{_{\frac{N}{sp},\Omega}}|\|u\||_{_{\mathbb{W}_{p,q}(\Omega,\Gamma)}}^{p-1}|\|v\||_{_{\mathbb{W}_{p,q}(\Omega,\Gamma)}}
    \end{eqnarray}
for some constant $C_1>0$, and similarly an application of Theorem \ref{besov-embedding} gives
\begin{equation}\label{1.55}
\left | \displaystyle\int_{\Gamma} \beta |u|^{q-2}uv \, \mathrm{d}\mu_{x} \right |
 \,\leq\,  \|\beta \|_{_{\frac{d}{\eta q},\Gamma}} \| u \|_{_{\frac{dq}{d-\eta q},\Gamma}}^{q-1} \| v \|_{_{\frac{dq}{d-\eta q},\Gamma}} \ \leq  \  C_{2}\|\beta \|_{_{\frac{d}{\eta q}}}|\|u\||_{_{\mathbb{W}_{p,q}(\Omega,\Gamma)}}^{q-1}|\|v\||_{_{\mathbb{W}_{p,q}(\Omega,\Gamma)}}
\end{equation}
for some $C_2>0$.
Then, by combining the inequalities \eqref{1.52}, \eqref{1.53}, \eqref{1.54} and \eqref{1.55}, we get that  
$$| \mathcal{E}_{p,q}(u,v)|\,\leq\, 4 \max \left \{1,C_{N,p,s}, C_{1}\| \alpha \|_{_{\frac{N}{sp}}}, C_{2}\|\beta \|_{_{\frac{d}{\eta q}}} \right \} \max \{|\|u\||_{_{\mathbb{W}_{p,q}(\Omega,\Gamma)}}^{p-1},|\|u\||_{_{\mathbb{W}_{p,q}(\Omega,\Gamma)}}^{q-1}   \}|\|v\||_{_{\mathbb{W}_{p,q}(\Omega,\Gamma)}}$$
for all $v \in \mathbb{W}_{p,q}(\Omega,\Gamma)$. The above estimate shows that $\mathcal{E}_{p,q}(u,\cdot) \in \mathbb{W}_{p,q}(\Omega,\Gamma)^{\ast}$ for each $u \in \mathbb{W}_{p,q}(\Omega,\Gamma)$. Using the fact that the norm function in every Banach space is continuous, we get that  
\begin{equation*}
    \lim_{t \rightarrow 0} \mathcal{E}_{p,q}(u+tw,v)=\mathcal{E}_{p,q}(u,v), \qquad \forall u,v,w \in \mathbb{W}_{p,q}(\Omega,\Gamma)
\end{equation*}
which shows that $\mathcal{E}_{p,q}$ is hemicontinuous. Furthermore, recalling the positivity of the coefficients $\alpha,\,\beta$, a direct application of \eqref{r>1} clearly shows that
$$\mathcal{E}_{p,q}(u,u-v)-\mathcal{E}_{p,q}(v,u-v)>0 \indent\,\,\,\textrm{for all}\,\,u,\,v\in\mathbb{W}_{p,q}(\Omega,\Gamma)\,\,\textrm{with}\,\,u\neq v,$$
Proving that $\mathcal{E}_{p,q}(\cdot ,\cdot)$ is strictly monotone. Now, by the definition of $\mathcal{E}_{p,q}(\cdot, \cdot )$,  and  the properties of $\alpha$ and $\beta$, we have that 
\begin{eqnarray}\label{1.56}
 \mathcal{E}_{p,q}(u,u)  
 & \geq & C_{N,p,s} \int_{\Omega}\int_{\Omega}\frac{|u(x)-u(y)|^{p}}{|x-y|^{N+sp}}\, \mathrm{d}x \, \mathrm{d}y + \alpha_{0} \int_{\Omega}|u|^{p}\, \mathrm{d}x \nonumber \\ \nonumber \\
&&+\int_{\Gamma}\int_{\Gamma}\frac{|u(x)-u(y)|^{q}}{|x-y|^{d+\eta q}}\, \mathrm{d}\mu_{x} \, \mathrm{d}\mu_{y}+ \beta_{0} \int_{\Gamma} |u|^{q}\mathrm{d}\mu_{x} 
\,\geq\, M_{0} \left (\| u \|_{_{W^{s,p}(\Omega)}}^{p}+\| u \|_{_{\mathbb{B}_{\eta}^{q}(\Gamma)}}^{q} \right )
\end{eqnarray}
where 
\begin{equation}\label{1.60}
    M_{0}:=\min \{1,\alpha_{0},\beta_{0}, C_{N,p,s}\}>0.
\end{equation}
Taking $\theta:= \min \{ p,q \}$, $\theta > 1$, it follows that
\begin{equation}\label{1.61}
    |\|u\||_{_{\mathbb{W}_{p,q}(\Omega,\Gamma)}}^{\theta}= \left ( \| u \|_{_{W^{s,p}(\Omega)}}+\| u \|_{_{\mathbb{B}_{\eta}^{q}(\Gamma)}} \right )^{\theta} \leq K \left ( \| u \|_{_{W^{s,p}(\Omega)}}^{\theta}+\| u \|_{_{\mathbb{B}_{\eta}^{q}(\Gamma)}}^{\theta} \right ).
\end{equation}
We then distinguish two cases:\\

\noindent$\bullet$ \textbf{\underline{Case 1}}. If $p>q$, applying Young's inequality, we have that 
\begin{equation}
    \| u \|_{_{W^{s,p}(\Omega)}}^{q}=\| u \|_{_{W^{s,p}(\Omega)}}^{q} \times 1 \leq \frac{\| u \|_{_{W^{s,p}(\Omega)}}^{q \left ( p/q\right )}}{p/q} + \frac{1^{\left ( p/q\right )'}}{\left ( p/q\right )'} \leq \| u \|_{_{W^{s,p}(\Omega)}}^{p} + 1.
\end{equation}
Combining the above inequality with \eqref{1.61}, it follows that
\begin{equation*}
    |\|u\||_{_{\mathbb{W}_{p,q}(\Omega,\Gamma)}}^{\theta} \leq  K \left ( \| u \|_{_{W^{s,p}(\Omega)}}^{q}+\| u \|_{_{\mathbb{B}_{\eta}^{q}(\Gamma)}}^{q} \right ) \leq K \left ( \| u \|_{_{W^{s,p}(\Omega)}}^{p}+\| u \|_{_{\mathbb{B}_{\eta}^{q}(\Gamma)}}^{q} +1 \right ).
\end{equation*}

\noindent$\bullet$ \textbf{\underline{Case 2}}. If $p<q$, applying Young's inequality, we have that  
\begin{equation}
    \| u \|_{_{\mathbb{B}_{\eta}^{q}(\Gamma)}}^{p}=\| u \|_{_{\mathbb{B}_{\eta}^{q}(\Gamma)}}^{p} \times 1 \leq \frac{\| u \|_{_{\mathbb{B}_{\eta}^{q}(\Gamma)}}^{p \left ( q/p\right )}}{q/p} + \frac{1^{\left ( q/p\right )'}}{\left ( q/p\right )'} \leq \| u \|_{_{\mathbb{B}_{\eta}^{q}(\Gamma)}}^{q} + 1,
\end{equation}
and thus from the above inequality with \eqref{1.61}, it follows that
\begin{equation*}
|\|u\||_{_{\mathbb{W}_{p,q}(\Omega,\Gamma)}}^{\theta} \leq  K \left ( \| u \|_{_{W^{s,p}(\Omega)}}^{p}+\| u \|_{_{\mathbb{B}_{\eta}^{q}(\Gamma)}}^{p} \right ) \leq K \left ( \| u \|_{_{W^{s,p}(\Omega)}}^{p}+\| u \|_{_{\mathbb{B}_{\eta}^{q}(\Gamma)}}^{q} +1 \right ).
\end{equation*}
Therefore, 
\begin{equation}\label{1.62}
|\|u\||_{_{\mathbb{W}_{p,q}(\Omega,\Gamma)}}^{\theta}  \leq K \left ( \| u \|_{_{W^{s,p}(\Omega)}}^{p}+\| u \|_{_{\mathbb{B}_{\eta}^{q}(\Gamma)}}^{q} +1 \right ).
\end{equation}
From \eqref{1.56} and \eqref{1.62}, we deduce that 
$$\mathcal{E}_{p,q}(u,u) \geq  M_{0} \left (\| u \|_{_{W^{s,p}(\Omega)}}^{p}+\| u \|_{_{\mathbb{B}_{\eta}^{q}(\Gamma)}}^{q} \right )= M_{0} \left (\| u \|_{_{W^{s,p}(\Omega)}}^{p}+\| u \|_{_{\mathbb{B}_{\eta}^{q}(\Gamma)}}^{q} +1 \right )-M_{0}\geq \frac{M_{0}}{K}|\|u\||_{_{\mathbb{W}_{p,q}(\Omega,\Gamma)}}^{\theta} - M_{0}.$$
Since $\theta > 1$, the above estimate implies that
\begin{equation*}
    \frac{\mathcal{E}_{p,q}(u,u)}{|\|u\||_{_{\mathbb{W}_{p,q}(\Omega,\Gamma)}}} \rightarrow \infty, \qquad \text{as}  \qquad |\|u\||_{_{\mathbb{W}_{p,q}(\Omega,\Gamma)}} \rightarrow \infty,
\end{equation*}
proving that $\mathcal{E}_{p,q}(\cdot, \cdot )$ is coercive.
\end{proof}

Now, we will prove the existence and uniqueness of weak solutions to the problem \eqref{ep}.

\begin{theorem}\label{solvability}
    Let $(f,g) \in \mathbb{X}^{r,l}(\Omega,\Gamma)$, where $r \geq (p^{*})'$ and $l \geq (q^{*})'$ (for $(p^{*})',(q^{*})'$ given by \eqref{critical-number}). Then the boundary value problem \eqref{ep} admits a unique weak solution in the sense that there exists a unique function $u \in \mathbb{W}_{p,q}(\Omega,\Gamma)$ such that
\begin{equation}\label{ws}
    \mathcal{E}_{p,q}(u, \phi)=\int
    _{\Omega} f \phi\, \mathrm{d}x + \int_{\Gamma} g \phi\, \mathrm{d}\mu_{x}, \qquad \forall \phi \in \mathbb{W}_{p,q}(\Omega,\Gamma),
\end{equation}
where $\mathcal{E}_{p,q}(\cdot,\cdot)$ is defined as in \eqref{form}.
\end{theorem}

\begin{proof}
The conclusions of Lemma \ref{form-prop} imply particularly that for every $u \in \mathbb{W}_{p,q}(\Omega,\Gamma)$, there exists an operator $\mathcal{L}(u) \in \mathbb{W}_{p,q}(\Omega,\Gamma)^{\ast}$ such that 
\begin{equation}\label{1.58}
        \mathcal{E}_{p,q}(u, \phi)=\langle \mathcal{L}(u), \phi\rangle_{_{\mathbb{W}_{p,q}}}, \qquad \forall \phi \in \mathbb{W}_{p,q}(\Omega,\Gamma),
    \end{equation}
where $\langle \cdot, \cdot \rangle_{_{\mathbb{W}_{p,q}}}$ denotes the duality between $\mathbb{W}_{p,q}(\Omega,\Gamma)$ and $\mathbb{W}_{p,q}(\Omega,\Gamma)^{*}$. It follows that the equality \eqref{1.58} defines an operator $\mathcal{L}:\mathbb{W}_{p,q}(\Omega,\Gamma) \longrightarrow \mathbb{W}_{p,q}(\Omega,\Gamma)^{\ast}$ which by Lemma \ref{form-prop} is also hemicontinuous, bounded, strictly monotone, and coercive, and it follows from \cite[Corollary 2.2, pag. 39]{showalter2013monotone} that $\mathcal{L}$ is surjective. Hence, for every $T \in \mathbb{W}_{p,q}(\Omega,\Gamma)^{\ast}$, there exists $u \in \mathbb{W}_{p,q}(\Omega,\Gamma)$ satisfying
\begin{equation}
     \langle \mathcal{L}(u), \phi \rangle_{\mathbb{W}_{p,q}}=\mathcal{E}_{p,q}(u, \phi)=\langle T, \phi\rangle_{\mathbb{W}_{p,q}}, \qquad \forall \phi\in \mathbb{W}_{p,q}(\Omega,\Gamma).
\end{equation}
Because $\mathcal{L}$ is strictly monotone, Browder's Theorem (e.g. \cite[Theorem 5.3.22]{drabek2007methods}) implies that the solution $u \in \mathbb{W}_{p,q}(\Omega,\Gamma)$ of \eqref{ep} is unique. To complete the proof, given $(f,g) \in \mathbb{X}^{r,l}(\Omega,\Gamma)$ as in the theorem, define the operator $T: \mathbb{W}_{p,q}(\Omega,\Gamma) \longrightarrow \mathbb{R} $, by 
\begin{equation}\label{1.73}
    Tu:=\int
    _{\Omega} f u \, \mathrm{d}x + \int_{\Gamma} g u \, \mathrm{d}\mu_{x}, \qquad \forall u \in \mathbb{W}_{p,q}(\Omega,\Gamma).
\end{equation}
A straightforward application of H\"older's inequality together with Theorem \ref{sobolev-embedding} and Theorem \ref{besov-embedding} shows that 
   $$ |Tu|\,\leq\,C \left (  \| f \|_{_{r,\Omega}}  + \| g \|_{_{l,\Gamma}}   \right )|\| u \||_{_{\mathbb{W}_{p,q}(\Omega,\Gamma)}},$$ for some constant $C>0$, proving that
$T \in \mathbb{W}_{p,q}(\Omega,\Gamma)^{\ast} $, which completes the proof.
\end{proof}

\section{A priori estimates}\label{sec4}
 
\indent In this section we are fully devoted in establishing $L^{\infty}$-norm estimates for weak solutions to problem \eqref{ep}, under all the conditions in Assumption \eqref{A}. To achieve this, we employ a variation of De Giorgi's techniques, motivated by the results obtained in \cite{creo2021fractional,diaz2022generalized,velez2015global}. Since the nonlocal structure of the problem brings substantial complications, complete proofs will be given for the main results of the section.\\
\indent To begin, given $u \in \mathbb{W}_{p,q}(\Omega,\Gamma)$ and $k \geq 0 $ a fixed real number, we put 
\begin{equation}\label{1.26}
    \widehat{u}_{k}:=(|u|-k)^{+}\textrm{sgn}(u).
\end{equation}
Since it is well known that the spaces $W^{s,p}(\Omega)$ and $\mathbb{B}_{\eta}^{q}(\Gamma)$ are Banach lattices, the same conclusion holds for the space $\mathbb{W}_{p,q}(\Omega,\Gamma)$, and consequently $\widehat{u}_{k}\in \mathbb{W}_{p,q}(\Omega,\Gamma)$.

For a set $A \subseteq \overline{\Omega}$, we put

\begin{equation*}
    A_{k}:=\{  x \in A \ | \ |u| \geq k \} , \qquad  \widetilde{A}_{k}:=A \backslash A_{k} = \{  x \in A \ | \ |u| < k \},
\end{equation*}
\begin{equation}\label{1.27}
    A_{k}^{+}:=\{  x \in A \ | \ u \geq k \}, \qquad A_{k}^{-}:=\{  x \in A \ | \ u \leq -k \}, 
\end{equation}
\begin{equation}\label{1.27a}
    \widetilde{A}_{k}^{+}:=\{x \in A \ | \ 0 \leq  u < k \}, \qquad \widetilde{A}_{k}^{-}:=\{x \in A \ | \ -k < u < 0\}.
\end{equation}
Clearly $A_{k}=A_{k}^{+} \cup A_{k}^{-}$ and $\widetilde{A}_{k}=\widetilde{A}_{k}^{+} \cup \widetilde{A}_{k}^{-}$. We now collect few properties which will be used in the main result of this section. The proofs are straightforward (although tedious), so we will skip it (for a proof, refer to \cite[Lemma 5.2]{MESINO2024}).

\begin{lemma}\label{desigualdades_estimadas}
     Let $u, u_{1}, u_{2}\in \mathbb{W}_{p,q}(\Omega,\Gamma)$; $\lambda \in (1,\infty)$, and $w=u_{1}-u_{2}$. For $\mathcal{D}$ denoting either $\Omega$ or $\Gamma$,the following hold:
     \begin{enumerate}
         \item[(a)] $|\widehat{u}_{k}| \leq |u|$, in $\mathcal{D}$.
         \item[(b)] $|\widehat{u}_{k}(x)-\widehat{u}_{k}(y)|\leq |u(x)-u(y)|$ in $\mathcal{D} \times \mathcal{D}$.
         \item[(c)] $\left [ |u_{1}|^{\lambda-2}u_{1}-|u_{2}|^{\lambda-2}u_{2} \right ] \widehat{w}_{k}\geq 0$, in $\mathcal{D}$.
         \item[(d)] $\left [|u_{1}(x)-u_{1}(y)|^{\lambda-2}(u_{1}(x)-u_{1}(y))-|u_{2}(x)-u_{2}(y)|^{\lambda-2}(u_{2}(x)-u_{2}(y))\right ]$\\
         \hspace{2cm}$\times(\widehat{w}_{k}(x)-\widehat{w}_{k}(y))\geq 0$, in $\mathcal{D} \times \mathcal{D}$.
         \item[(e)] If $\lambda \geq 2$, then $\left [ |u_{1}|^{\lambda-2}u_{1}-|u_{2}|^{\lambda-2}u_{2} \right ] \widehat{w}_{k}\geq C_{0}|\widehat{w}_{k}|^{\lambda}$, in $\mathcal{D}$.
         \item [(f)] If $\lambda \geq 2$, then \\
         $\left[|u_{1}(x)-u_{1}(y)|^{\lambda-2}(u_{1}(x)-u_{1}(y))-|u_{2}(x)-u_{2}(y)|^{\lambda-2}(u_{2}(x)-u_{2}(y))\right ]$\\
         \hspace{2cm}$\times (\widehat{w}_{k}(x)-\widehat{w}_{k}(y)) \geq C_{0}|\widehat{w}_{k}(x)-\widehat{w}_{k}(y)|^{\lambda}, \quad \text{in} \quad \mathcal{D} \times  \mathcal{D}.$
\item[(g)] $\mathcal{E}_{p,q}(\widehat{u}_{k}, \widehat{u}_{k}) \leq \mathcal{E}_{p,q}(u,\widehat{u}_{k}).$
\item[(h)] $\mathcal{E}_{p,q}(u_{1}, \widehat{w}_{k})-\mathcal{E}_{p,q}(u_{2}, \widehat{w}_{k}) \leq \mathcal{E}_{p,q}(u_{1},u_{1}-u_{2})-\mathcal{E}_{p,q}(u_{2},u_{1}-u_{2})$.
\item[(i)] $ \mathcal{E}_{p,q}(u_{1}, \widehat{w}_{k})-\mathcal{E}_{p,q}(u_{2}, \widehat{w}_{k})\leq 4 \left [ \mathcal{E}_{p,q}(u_{1}, u_{1})+\mathcal{E}_{p,q}(u_{2}, u_{2}) \right ]$. 
     \end{enumerate}
\end{lemma}

Now, we are ready to give the main result of this section.

\begin{theorem}\label{global-boundedness}
Let $(f_{i},g_{i})\in \mathbb{X}^{r_i,l_i}(\Omega,\Gamma)$ for $i\in\{0,1,2\}$, where 
\begin{equation}\label{1.29}
    r_i >\frac{N}{sp}\chi_{_{\{i=0\}}}+\left(\left \{ \begin{array}{ccc}
        \frac{Np}{Np+2sp-2N},& \text{if} & \frac{2N}{N+2s} < p <2 \\ \\
        \frac{N}{sp}, &\text{if} & p \geq 2 
    \end{array}   \right.\right)\chi_{_{\{1\leq i\leq 2\}}}
\end{equation}
and
\begin{equation}\label{1.30}
    l_i >\frac{d}{\eta q}\chi_{_{\{i=0\}}}+\left(\left \{ \begin{array}{ccc}
        \frac{dq}{dq+2\eta q-2d},& \text{if} & \frac{2d}{d+2\eta} < q <2 \\ \\
        \frac{d}{\eta q}, &\text{if} & q \geq 2 
    \end{array}   \right.\right)\chi_{_{\{1\leq i\leq 2\}}}.
\end{equation}
\begin{enumerate}
\item[(a)] If $u \in \mathbb{W}_{p,q}(\Omega,\Gamma)$  is a weak solution of problem \eqref{ep} related to $(f_0,g_0)$, then there exists a constant $\mathfrak{C}_0>0$ (independent of $u$), such that 
  \begin{equation*}\label{bounded-simple}
    \max \{\|u\|_{_{\infty,\Omega}},\|u\|_{_{\infty,\Gamma}} \} \leq\mathfrak{C}_0 \left [   \|f_0\|_{_{r_0,\Omega}}^{\frac{1}{p-1}}   +  \|f_0\|_{_{r_0,\Omega}}^{\frac{p}{q(p-1)}}+\|g_0\|_{_{l_0,\Gamma}}^{\frac{1}{q-1}}  +\|g_0\|_{_{l_0,\Gamma}}^{\frac{q}{p(q-1)}} \right ].
\end{equation*}
\item[(b)] If $u_{1},u_{2} \in \mathbb{W}_{p,q}(\Omega,\Gamma)$  are weak solutions of problem \eqref{ep} related to $(f_{1},g_{1})$ and  $(f_{2},g_{2})$, respectively, then there exists a constant $\mathfrak{C}:=\mathfrak{C}(p,q,f_{1},f_{2},g_{1},g_{2},|\Omega|,\mu(\Gamma)) >0$ (independent of $u_{1}$ and $u_{2}$), such that\\[2ex]
\indent$\max \{\|u_{1}-u_{2}\|_{_{\infty,\Omega}},\|u_{1}-u_{2}\|_{_{\infty,\Gamma}} \} $\\
\begin{equation}\label{bounded-difference}
\begin{array}{lll}
     &\leq&   \mathfrak{C} \left [ \chi_{_{B}}\left ( \|f_{1}-f_{2}\|_{_{r,\Omega}}^{\frac{1}{p-1}}+ \|f_{1}-f_{2}\|_{_{r,\Omega}}^{\frac{p}{q(p-1)}} +\|g_{1}-g_{2}\|_{_{l,\Gamma}}^{\frac{1}{q-1}} +\|g_{1}-g_{2}\|_{_{l,\Gamma}}^{\frac{q}{p(q-1)}} \right )  \right.  \nonumber \\
    &&\left.  \qquad \quad + \chi_{_{A}}\left (|\|(f_{1}-f_{2},g_{1}-g_{2})\||_{_{r,l}} + |\|(f_{1}-f_{2},g_{1}-g_{2})\||_{_{r,l}}^{\frac{\max\{p,q\}}{\min\{p,q\}}} \right ) \right ],
\end{array}
\end{equation}
where 
\begin{equation*}
    A:= \left \{ \frac{2N}{N+2s} < p <2  \quad \textrm{or} \quad  \frac{2d}{d+2\eta} < q <2 \right \} \quad \textrm{and} \quad   B:= \left \{ p \geq 2  \quad \textrm{or} \quad  q \geq 2 \right \},
\end{equation*}
for $r:=\min\{r_1,r_2\}$ and $l:=\min\{l_1,l_2\}$.
\end{enumerate}
\end{theorem}
\begin{proof}
We will only prove part (b) (since part (a) follows in an easier manner). Let $(f_{1},g_{1}), (f_{2},g_{2}) \in \mathbb{X}^{r,l}(\Omega,\Gamma)$ be as in the theorem, for $r:=\min\{r_1,r_2\}$ and $l:=\min\{l_1,l_2\}$ ($r_i,\,l_i$ are given by \eqref{1.29} and \eqref{1.30}, respectively, for $i\in\{1,2\}$), and let $u_{1},\,u_{2} \in \mathbb{W}_{p,q}(\Omega,\Gamma)$ be solutions to \eqref{ep} related to $(f_{1},g_{1})$ and $(f_{2},g_{2})$, respectively. Given $w:=u_{1}-u_{2}$, let $k\geq k_{0}$ be a real number, for $k_{0}\geq 0$, and let $\widehat{w}_{k} \in \mathbb{W}_{p,q}(\Omega,\Gamma)$ be defined by \eqref{1.26}. Given that the exponents p and q are unrelated to each other, we divide the proof into four cases. \\
$\bullet$ \underline{\textbf{Case 1}}: Assume that $\frac{2N}{N+2s} < p <2$ and $\frac{2d}{d+2\eta} < q <2$.  For $\epsilon \in (0,1]$, let 
$$C_{p,\epsilon}:=\min\{ c_{p}, c_{p}^{*},c'_{p},c_{p,\epsilon} \}>0 \qquad \text{and}  \qquad C_{q,\epsilon}:=\min\{ c_{q}, c_{q}^{*},c'_{q},c_{q,\epsilon} \}>0,$$ where $c_{p}, c_{p}^{*},c'_{p},c_{p,\epsilon},c_{q}, c_{q}^{*},c'_{q},c_{q,\epsilon} $ denote the constants in Proposition \ref{duality} (for $r=p$ or $r=q$). Then there exists a constant $C_{p,q}>0$ such that
\begin{equation}\label{ce}
    C_{\epsilon}(p,q):=\min \{ C_{p,\epsilon},C_{q,\epsilon}\}\geq \frac{\max \left  \{ \epsilon^{2-p}, \epsilon^{2-q} \right \}}{C_{p,q}}, \qquad \text{for all} \ \epsilon \in (0,1].
\end{equation}
We will assume that $p>q$; the reverse inequality follows similarly, and the equality is the easier case.\\
Then we define the following sets:  \\

$ \Omega(\epsilon):=\{ x \in \Omega \ | \ |w(x)| \geq \epsilon \min \{ |u_{1}(x)|,|u_{2}(x)| \} \},$ \\ 

$ \Gamma(\epsilon):=\{ x \in \Gamma \ | \ |w(x)| \geq \epsilon^{p/q} \min \{ |u_{1}(x)|,|u_{2}(x)| \} \},$ \\

$\Omega^{*}(\epsilon):=\{ (x,y) \in \Omega \times \Omega \ | \ |w(x)-w(y)| \geq \epsilon \min \{ |u_{1}(x)-u_{1}(y)|,|u_{2}(x)-u_{2}(y)| \} \},$ \\

$\Gamma^{*}(\epsilon):=\{ (x,y) \in \Gamma \times \Gamma \ | \ |w(x)-w(y)| \geq \epsilon^{p/q} \min \{ |u_{1}(x)-u_{1}(y)|,|u_{2}(x)-u_{2}(y)| \} \},$ \\

$\Omega_{k,\epsilon}:=\Omega_{k} \cap \Omega(\epsilon), \qquad  \widetilde{\Omega}_{k,\epsilon}:=\Omega_{k}\backslash \Omega_{k,\epsilon},  \qquad \Gamma_{k,\epsilon}:=\Gamma_{k} \cap \Gamma(\epsilon), \qquad  \widetilde{\Gamma}_{k,\epsilon}:=\Gamma_{k}\backslash \Gamma_{k,\epsilon},$ \\ 

$\Omega^{*}_{k,\epsilon}:=(\Omega_{k} \times \Omega_{k} )  \cap \Omega^{*}(\epsilon), \qquad  \widetilde{\Omega}^{*}_{k,\epsilon}:=(\Omega \times \Omega )\backslash \Omega^{*}_{k,\epsilon},$ \\

$\Gamma^{*}_{k,\epsilon}:=(\Gamma_{k} \times \Gamma_{k} )  \cap \Gamma^{*}(\epsilon), \qquad  \widetilde{\Gamma}^{*}_{k,\epsilon}:=(\Gamma \times \Gamma )\backslash \Gamma^{*}_{k,\epsilon}. $\\

\noindent Then, we clearly see that 
$$  \Omega_{k}=\Omega_{k,\epsilon} \cup  \widetilde{\Omega}_{k,\epsilon}, \qquad \Gamma_{k}=\Gamma_{k,\epsilon} \cup  \widetilde{\Gamma}_{k,\epsilon},\qquad \Omega \times \Omega=\Omega^{*}_{k,\epsilon} \cup \widetilde{\Omega}^{*}_{k,\epsilon}, \qquad \Gamma \times \Gamma= \Gamma^{*}_{k,\epsilon} \cup \widetilde{\Gamma}^{*}_{k,\epsilon}.$$  
Applying Proposition \ref{duality} and Lemma \ref{desigualdades_estimadas} when necessary, for $\epsilon \in (0,1]$, to obtain the following:

\begin{flalign*}
    &\mathcal{E}_{p,q}(u_{1}, \widehat{w}_{k})-\mathcal{E}_{p,q}(u_{2}, \widehat{w}_{k})&
\end{flalign*}

\noindent$\geq C_{N,p,s}\displaystyle\int\limits_{\Omega_{k,\epsilon}^{*}} \displaystyle\int\limits_{\Omega_{k,\epsilon}^{*}} \frac{\left [|u_{1}(x)-u_{1}(y)|^{p-2}(u_{1}(x)-u_{1}(y))-|u_{2}(x)-u_{2}(y)|^{p-2}(u_{2}(x)-u_{2}(y))\right ](\widehat{w}_{k}(x)-\widehat{w}_{k}(y))}{|x-y|^{N+sp}}\, \mathrm{d}x \, \mathrm{d}y $ \\ 
    
$\hspace{15pt}+ \displaystyle\int\limits_{\Omega_{k,\epsilon}} \alpha \left [ |u_{1}|^{p-2}u_{1}-|u_{2}|^{p-2}u_{2} \right ] \widehat{w}_{k}\, \mathrm{d}x +   \displaystyle\int\limits_{\Gamma_{k,\epsilon}}\beta \left [ |u_{1}|^{q-2}u_{1}-|u_{2}|^{q-2}u_{2} \right ] \widehat{w}_{k} \, \mathrm{d}\mu_{x}$ \\
    
\noindent$\hspace{15pt}+\displaystyle\int\limits_{\Gamma_{k,\epsilon}^{*}} \displaystyle\int\limits_{\Gamma_{k,\epsilon}^{*}} \frac{\left [|u_{1}(x)-u_{1}(y)|^{q-2}(u_{1}(x)-u_{1}(y))-|u_{2}(x)-u_{2}(y)|^{q-2}(u_{2}(x)-u_{2}(y))\right ] (\widehat{w}_{k}(x)-\widehat{w}_{k}(y))}{|x-y|^{d+\eta q}}\, \mathrm{d}\mu_{x} \, \mathrm{d}\mu_{y}$ \\ 

 $\geq C_{\epsilon}(p,q) \left (
C_{N,p,s} \displaystyle\int\limits_{\Omega_{k,\epsilon}^{*}} \displaystyle\int\limits_{\Omega_{k,\epsilon}^{*}} \frac{|\widehat{w}_{k}(x)-\widehat{w}_{k}(y)|^{p}}{|x-y|^{N+sp}}\, \mathrm{d}x \, \mathrm{d}y +\displaystyle\int\limits_{\Omega_{k,\epsilon}} \alpha |\widehat{w}_{k}|^{p}\mathrm{d}x  +   \displaystyle\int\limits_{\Gamma_{k,\epsilon}}\beta  |\widehat{w}_{k}|^{q} \,  \mathrm{d}\mu_{x}\right.$ \\ 

$\hspace{15pt} \left. + \displaystyle\int\limits_{\Gamma_{k,\epsilon}^{*}} \displaystyle\int\limits_{\Gamma_{k,\epsilon}^{*}} \frac{|\widehat{w}_{k}(x)-\widehat{w}_{k}(y)|^{q}}{|x-y|^{d+\eta q}}\, \mathrm{d}\mu_{x} \, \mathrm{d}\mu_{y} \right )$\\ \\ 

$\geq C_{\epsilon}(p,q) \left ( \mathcal{E}_{p,q}(\widehat{w}_{k}, \widehat{w}_{k})-
C_{N,p,s} \displaystyle\int\limits_{\widetilde{\Omega}_{k,\epsilon}^{*}} \displaystyle\int\limits_{\widetilde{\Omega}_{k,\epsilon}^{*}} \frac{|\widehat{w}_{k}(x)-\widehat{w}_{k}(y)|^{p}}{|x-y|^{N+sp}}\, \mathrm{d}x \, \mathrm{d}y -   \displaystyle\int\limits_{\widetilde{\Omega}_{k,\epsilon}} \alpha |\widehat{w}_{k}|^{p}\, \mathrm{d}x  \right.$\\ 

$\hspace{15pt}\left. -\displaystyle\int\limits_{\widetilde{\Gamma}_{k,\epsilon}^{*}} \displaystyle\int\limits_{\widetilde{\Gamma}_{k,\epsilon}^{*}} \frac{|\widehat{w}_{k}(x)-\widehat{w}_{k}(y)|^{q}}{|x-y|^{d+\eta q}}\, \mathrm{d}\mu_{x} \, \mathrm{d}\mu_{y} -   \displaystyle\int\limits_{\widetilde{\Gamma}_{k,\epsilon}}\beta  |\widehat{w}_{k}|^{q}  \mathrm{d}\mu_{x} \right )$ \\ 

$\geq C_{\epsilon}(p,q) \left ( \mathcal{E}_{p,q}(\widehat{w}_{k}, \widehat{w}_{k})-
\epsilon^{p}C_{N,p,s} \displaystyle\int\limits_{\widetilde{\Omega}_{k,\epsilon}^{*}} \displaystyle\int\limits_{\widetilde{\Omega}_{k,\epsilon}^{*}} \frac{|u_{1}(x)-u_{1}(y)|^{p}+|u_{2}(x)-u_{2}(y)|^{p}}{|x-y|^{N+sp}}\, \mathrm{d}x \, \mathrm{d}y  \right.$ \\ 

$\hspace{15pt} -   \epsilon^{p}\displaystyle\int\limits_{\widetilde{\Omega}_{k,\epsilon}} \alpha (|u_{1}|^{p}+|u_{2}|^{p}) \, \mathrm{d}x -\epsilon^{p}  \displaystyle\int\limits_{\widetilde{\Gamma}_{k,\epsilon}^{*}} \displaystyle\int\limits_{\widetilde{\Gamma}  _{k,\epsilon}^{*}} \frac{|u_{1}(x)-u_{1}(y)|^{q}+|u_{2}(x)-u_{2}(y)|^{q}}{|x-y|^{d+\eta q}}\, \mathrm{d}\mu_{x} \, \mathrm{d}\mu_{y} $ \\

$\left.\hspace{15pt}  -\epsilon^{p} \displaystyle\int\limits_{\widetilde{\Gamma}_{k,\epsilon}}\beta(|u_{1}|^{q}+|u_{2}|^{q}) \, \mathrm{d}\mu_{x}  \right )$\\
   

$\geq\,C_{\epsilon}(p,q) \left [ \mathcal{E}_{p,q}(\widehat{w}_{k}, \widehat{w}_{k})-
\epsilon^{p}(\mathcal{E}_{p,q}(u_{1}, u_{1})+\mathcal{E}_{p,q}(u_{2}, u_{2})) \right ]$.\\[2ex]
By rearranging in the above calculations, we obtain that
\begin{equation}\label{F01}
\mathcal{E}_{p,q}(\widehat{w}_{k}, \widehat{w}_{k})
 \leq \frac{1}{C_{\epsilon}(p,q) }  \left[ \mathcal{E}_{p,q}(u_{1}, \widehat{w}_{k})-\mathcal{E}_{p,q}(u_{2}, \widehat{w}_{k}) \right ] +\epsilon^{p}\left [ \mathcal{E}_{p,q}(u_{1}, u_{1})+\mathcal{E}_{p,q}(u_{2}, u_{2}) \right ],
 \end{equation}
 and by virtue of \eqref{ce}, equation \eqref{F01} becomes
 \begin{equation}\label{F02}
\mathcal{E}_{p,q}(\widehat{w}_{k}, \widehat{w}_{k})
 \leq \displaystyle\frac{C_{p,q}}{\epsilon^{2-p}}  \left[ \mathcal{E}_{p,q}(u_{1}, \widehat{w}_{k})-\mathcal{E}_{p,q}(u_{2}, \widehat{w}_{k}) \right ] +\epsilon^{p}\left [ \mathcal{E}_{p,q}(u_{1}, u_{1})+\mathcal{E}_{p,q}(u_{2}, u_{2}) \right ].
 \end{equation}
Applying Proposition \ref{xita} for 
$$r:=p,\qquad \qquad \tau:= C_{p,q}, \qquad \qquad \varrho:=2-p,\qquad \qquad \xi:= \mathcal{E}_{p,q}(\widehat{w}_{k}, \widehat{w}_{k}),$$ $$\varsigma:=\mathcal{E}_{p,q}(u_{1}, \widehat{w}_{k})-\mathcal{E}_{p,q}(u_{2}, \widehat{w}_{k}), \qquad \mbox{and} \qquad c:=\mathcal{E}_{p,q}(u_{1}, u_{1})+\mathcal{E}_{p,q}(u_{2}, u_{2}), $$
it follows that\\[2ex]
$\mathcal{E}_{p,q}(\widehat{w}_{k}, \widehat{w}_{k})\leq(C_{p,q}+1) \left[ \mathcal{E}_{p,q}(u_{1}, \widehat{w}_{k})-\mathcal{E}_{p,q}(u_{2}, \widehat{w}_{k}) \right ]^{\frac{p}{2}}\left [ \mathcal{E}_{p,q}(u_{1}, u_{1})+\mathcal{E}_{p,q}(u_{2}, u_{2}) \right ]^{\frac{2-p}{2}} $\\
\begin{equation}\label{F03}
+ (C_{p,q}+1) \left[ \mathcal{E}_{p,q}(u_{1}, \widehat{w}_{k})-\mathcal{E}_{p,q}(u_{2}, \widehat{w}_{k}) \right ]^{\frac{p}{2}} \left[ \mathcal{E}_{p,q}(u_{1}, \widehat{w}_{k})-\mathcal{E}_{p,q}(u_{2}, \widehat{w}_{k}) \right ]^{\frac{2-p}{2}}.
\end{equation}
Combining \eqref{F03} with Lemma \ref{desigualdades_estimadas}, we get that\\

$\mathcal{E}_{p,q}(\widehat{w}_{k}, \widehat{w}_{k})\leq (C_{p,q}+1) \left[ \mathcal{E}_{p,q}(u_{1}, \widehat{w}_{k})-\mathcal{E}_{p,q}(u_{2}, \widehat{w}_{k}) \right ]^{\frac{p}{2}}\left [ \mathcal{E}_{p,q}(u_{1}, u_{1})+\mathcal{E}_{p,q}(u_{2}, u_{2}) \right ]^{\frac{2-p}{2}}$\\
$$+ 4^{\frac{2-p}{2}}(C_{p,q}+1) \left[ \mathcal{E}_{p,q}(u_{1}, \widehat{w}_{k})-\mathcal{E}_{p,q}(u_{2}, \widehat{w}_{k}) \right ]^{\frac{p}{2}} \left [ \mathcal{E}_{p,q}(u_{1}, u_{1})+\mathcal{E}_{p,q}(u_{2}, u_{2}) \right ]^{\frac{2-p}{2}}.$$

\noindent Taking $C'=C'(p,q):=(C_{p,q}+1)(1+2^{2-p})\geq 0$ in the above inequality, we get that
\begin{equation}\label{1.1}
    \ \mathcal{E}_{p,q}(\widehat{w}_{k}, \widehat{w}_{k}) \leq C' \left[ \mathcal{E}_{p,q}(u_{1}, \widehat{w}_{k})-\mathcal{E}_{p,q}(u_{2}, \widehat{w}_{k}) \right ]^{\frac{p}{2}}\left [ \mathcal{E}_{p,q}(u_{1}, u_{1})+\mathcal{E}_{p,q}(u_{2}, u_{2}) \right ]^{\frac{2-p}{2}}.
\end{equation}
Also, note that from Theorem \ref{sobolev-embedding}, Theorem \ref{besov-embedding}, and the definition of the space $\mathbb{W}_{p,q}(\Omega,\Gamma)$, we have that
\begin{equation}\label{1.2}
    |\|v\||_{_{p^{*},q^{*}}}\,\leq\,c'|\|v\||_{_{\mathbb{W}_{p,q}(\Omega,\Gamma)}},
\end{equation}
for $c'=\max\{c_{1},c_{2}\}$. Since $u_{1},\,u_{2} \in \mathbb{W}_{p,q}(\Omega,\Gamma)$ solve \eqref{ep} related to $(f_{1},g_{1})$ and $(f_{2},g_{2})$, respectively, testing \eqref{ws} with $u_{1}$ and $u_{2}$, applying H\"older and Young inequalities together \eqref{1.2} and \eqref{1.56}, we can obtain the following calculation.\\
 
\noindent$\mathcal{E}_{p,q}(u_{1}, u_{1})+\mathcal{E}_{p,q}(u_{2}, u_{2})$ \\ \\
$\leq  \sum\limits_{i=1}^{2}\left ( \|f_{i}\|_{_{(p^{*})',\Omega}}||u_{i}||_{_{p^{*},\Omega}}  + \|g_{i}\|_{_{(q^{*})',\Gamma}}||u_{i}||_{_{q^{*},\Gamma}} \right )$\\ \\
$\leq c' \sum\limits_{i=1}^{2}\left ( \|f_{i}\|_{_{(p^{*})',\Omega}}  + \|g_{i}\|_{_{(q^{*})',\Gamma}} \right )|\|u_{i}\||_{_{\mathbb{W}_{p,q}(\Omega,\Gamma)}}$ \\ \\
$\leq \sum\limits_{i=1}^{2} \left [\left (  \frac{c'}{\epsilon^{1/p}} |\|(f_{i},g_{i})\||_{_{(p^{*})',(q^{*})'}} \right )^{\frac{p}{p-1}}   + \left (  \frac{c'}{\epsilon^{1/q}} |\|(f_{i},g_{i})\||_{_{(p^{*})',(q^{*})'}} \right )^{\frac{q}{q-1}}  \right ] +\epsilon \sum\limits_{i=1}^{2} \left [  ||u_{i}||_{_{W^{s,p}(\Omega)}}^{p} +  ||u_{i}||_{_{\mathbb{B}_{\eta}^{q}(\Gamma)}}^{q}   \right ]$\\ \\
$\leq \mathcal{K}_{\epsilon}(p,q) \sum\limits_{i=1}^{2} \left [|\|(f_{i},g_{i})\||_{_{(p^{*})',(q^{*})'}}^{\frac{p}{p-1}}   +  |\|(f_{i},g_{i})\||_{_{(p^{*})',(q^{*})'}}^{\frac{q}{q-1}}  \right ]+  \frac{\epsilon}{M_{0}} [\mathcal{E}_{p,q}(u_{1}, u_{1})+\mathcal{E}_{p,q}(u_{2}, u_{2})] ,$ \\ \\
for all $\epsilon >0$  and for some constant $\mathcal{K}_{\epsilon}(p,q)>0$. Letting
\begin{equation*}
    \mathcal{G}=\mathcal{G}_{p,q}(f_{1},f_{2},g_{1},g_{2}):=\mathcal{K}_{\epsilon}(p,q) \sum\limits_{i=1}^{2} \left [|\|(f_{i},g_{i})\||_{_{(p^{*})',(q^{*})'}}^{\frac{p}{p-1}}   +  |\|(f_{i},g_{i})\||_{_{(p^{*})',(q^{*})'}}^{\frac{q}{q-1}}  \right ]
\end{equation*}
and selecting $\epsilon>0$ suitably, we have 
\begin{equation}\label{1.3}
    \mathcal{E}_{p,q}(u_{1}, u_{1})+\mathcal{E}_{p,q}(u_{2}, u_{2}) \leq \mathcal{G}.
\end{equation}
Then, replacing the previous expression into equation \eqref{1.1} results into
\begin{equation}\label{1.4}
    \mathcal{E}_{p,q}(\widehat{w}_{k}, \widehat{w}_{k}) \leq C' \mathcal{G}^{\frac{2-p}{2}}\left[ \mathcal{E}_{p,q}(u_{1}, \widehat{w}_{k})-\mathcal{E}_{p,q}(u_{2}, \widehat{w}_{k}) \right ]^{\frac{p}{2}}.
\end{equation}
Letting
\begin{equation}\label{1.5}
\widetilde{\mathcal{G}}:=\widetilde{\mathcal{G}}_{p,q}(f_{1},f_{2},g_{1},g_{2})=\frac{C'(c')^{p/2}}{M_{0}} \mathcal{G}^{\frac{2-p}{2}},
\end{equation}
for $M_{0}>0$ and $c'>0$ denoting the constants in \eqref{1.60} and \eqref{1.2}, respectively, one deduces from \eqref{1.4} and \eqref{1.5} that 
\begin{equation}\label{1.6}
    \mathcal{E}_{p,q}(\widehat{w}_{k}, \widehat{w}_{k}) \leq \frac{M_{0}\widetilde{\mathcal{G}}}{(c')^{p/2}}\left[ \mathcal{E}_{p,q}(u_{1}, \widehat{w}_{k})-\mathcal{E}_{p,q}(u_{2}, \widehat{w}_{k}) \right ]^{\frac{p}{2}}.
\end{equation}
Now, by \eqref{1.56} and \eqref{1.6}  we have that
\begin{equation}\label{1.7}
    \|\widehat{w}_{k}\|_{_{W^{s,p}(\Omega)}}^{p}+\|\widehat{w}_{k}\|_{_{\mathbb{B}_{\eta}^{q}(\Gamma)}}^{q} \leq \frac{1}{M_{0}}\mathcal{E}_{p,q}(\widehat{w}_{k}, \widehat{w}_{k}) \leq \frac{\widetilde{\mathcal{G}}}{(c')^{p/2}}\left[ \mathcal{E}_{p,q}(u_{1}, \widehat{w}_{k})-\mathcal{E}_{p,q}(u_{2}, \widehat{w}_{k}) \right ]^{\frac{p}{2}}.
\end{equation}
Since $u_{1},u_{2} \in \mathbb{W}_{p,q}(\Omega,\Gamma)$ solve \eqref{ep} related to $(f_{1},g_{1})$ and $(f_{2},g_{2})$, respectively, testing \eqref{ws} with $\widehat{w}_{k}$ and applying H\"older inequality together \eqref{1.2}, proceeding as in the previous calculations we infer that
\begin{equation}\label{F04}
\mathcal{E}_{p,q}(u_{1}, \widehat{w}_{k})-\mathcal{E}_{p,q}(u_{2}, \widehat{w}_{k})\,\leq\,
c'\left|\left\|(\chi_{_{\Omega_k}},\chi_{_{\Gamma_k}})\right\|\right|_{_{\widetilde{r},\widetilde{l}}}|\|(f_{1}-f_{2},g_{1}-g_{2})\||_{_{r,l}}\left (\|\widehat{w}_{k}\|_{_{W^{s,p}(\Omega)}}+\|\widehat{w}_{k}\|_{_{\mathbb{B}_{\eta}^{q}(\Gamma)}} \right),
\end{equation}
where $\widetilde{r},\widetilde{l} \in (1, \infty)$ are such that $\frac{1}{\widetilde{r}}+ \frac{1}{r}+\frac{1}{p^{*}}=1$ and $\frac{1}{\widetilde{l}}+ \frac{1}{l}+\frac{1}{q^{*}}=1$. Then, substituting \eqref{F04} into \eqref{1.7} yields
\begin{equation}\label{1.9}
\|\widehat{w}_{k}\|_{_{W^{s,p}(\Omega)}}^{p}+\|\widehat{w}_{k}\|_{_{\mathbb{B}_{\eta}^{q}(\Gamma)}}^{q} 
    \,\leq\, \widetilde{\mathcal{G}}\left|\left\|(\chi_{_{\Omega_k}},\chi_{_{\Gamma_k}})\right\|\right|_{_{\widetilde{r},\widetilde{l}}}^{^{\frac{p}{2}}}|\|(f_{1}-f_{2},g_{1}-g_{2})\||_{_{r,l}}^{^{\frac{p}{2}}}\left (\|\widehat{w}_{k}\|_{_{W^{s,p}(\Omega)}}^{^{\frac{p}{2}}} +\|\widehat{w}_{k}\|_{_{\mathbb{B}_{\eta}^{q}(\Gamma)}}^{^{\frac{p}{2}}}  \right).
\end{equation}
Now, on the right side of inequality \eqref{1.9}, we must deal with the expression $\|\widehat{w}_{k}\|_{_{\mathbb{B}_{\eta}^{q}(\Gamma)}}^{\frac{p}{2}} $. For this, first note that combining Proposition \ref{desigualdades_estimadas}(i), \eqref{1.1}, and \eqref{1.3}, we get that 
\begin{equation}\label{1.10}
    \mathcal{E}_{p,q}(\widehat{w}_{k}, \widehat{w}_{k}) \leq C'(4\mathcal{G})^{\frac{p}{2}}\mathcal{G}^{\frac{2-p}{2}}= 2^{p}C'\mathcal{G}.
\end{equation}
Recalling our assumption  $p>q$, we apply \eqref{1.56} and \eqref{1.10} to get that
\begin{equation}\label{1.11}
\|\widehat{w}_{k}\|_{_{\mathbb{B}_{\eta}^{q}(\Gamma)}}^{^{\frac{p}{2}}}= \|\widehat{w}_{k}\|_{_{\mathbb{B}_{\eta}^{q}(\Gamma)}}^{^{\frac{p-q}{2}}}\|\widehat{w}_{k}\|_{_{\mathbb{B}_{\eta}^{q}(\Gamma)}}^{^{\frac{q}{2}}}\leq \left [ \frac{2^{p}C'\mathcal{G}}{M_{0}} \right]^{^{\frac{p-q}{2q}}}\|\widehat{w}_{k}\|_{_{\mathbb{B}_{\eta}^{q}(\Gamma)}}^{^{\frac{q}{2}}}:=
     \mathcal{F}\|\widehat{w}_{k}\|_{_{\mathbb{B}_{\eta}^{q}(\Gamma)}}^{^{\frac{q}{2}}}.
\end{equation}
Substituting \eqref{1.11} into \eqref{1.9}, and applying Young’s inequality, we deduce that\\[2ex]
$\|\widehat{w}_{k}\|_{_{W^{s,p}(\Omega)}}^{p}+\|\widehat{w}_{k}\|_{_{\mathbb{B}_{\eta}^{q}(\Gamma)}}^{q}$ \\ \\
$\leq \widetilde{\mathcal{G}}  \max\{1, \mathcal{F}\} \left|\left\|(\chi_{_{\Omega_k}},\chi_{_{\Gamma_k}})\right\|\right|_{_{\widetilde{r},\widetilde{l}}}^{^{\frac{p}{2}}}|\|(f_{1}-f_{2},g_{1}-g_{2})\||_{_{r,l}}^{^{\frac{p}{2}}} \|\widehat{w}_{k}\|_{_{W^{s,p}(\Omega)}}^{^{\frac{p}{2}}}$\\

$\hspace{15pt}+\widetilde{\mathcal{G}}  \max\{1, \mathcal{F}\}  \left|\left\|(\chi_{_{\Omega_k}},\chi_{_{\Gamma_k}})\right\|\right|_{_{\widetilde{r},\widetilde{l}}}^{^{\frac{p}{2}}}|\|(f_{1}-f_{2},g_{1}-g_{2})\||_{_{r,l}}^{^{\frac{p}{2}}}\|\widehat{w}_{k}\|_{_{\mathbb{B}_{\eta}^{q}(\Gamma)}}^{\frac{q}{2}}$  \\
$$\leq\,\mathcal{H}\left|\left\|(\chi_{_{\Omega_k}},\chi_{_{\Gamma_k}})\right\|\right|_{_{\widetilde{r},\widetilde{l}}}^{p}|\|(f_{1}-f_{2},g_{1}-g_{2})\||_{_{r,l}}^{p} + \frac{1}{2} \left (\|\widehat{w}_{k}\|_{_{W^{s,p}(\Omega)}}^{p} +\|\widehat{w}_{k}\|_{_{\mathbb{B}_{\eta}^{q}(\Gamma)}}^{q} \right ),$$
where 
\begin{equation*}
    \mathcal{H} =\mathcal{H}_{p,q}(f_{1},f_{2},g_{1},g_{2}):= (\widetilde{\mathcal{G}})^{2} \max \{ 
  1, \mathcal{F}  \}^{2}.
\end{equation*}
It follow that 
\begin{equation}\label{1.12}
    \|\widehat{w}_{k}\|_{_{W^{s,p}(\Omega)}}^{p} +\|\widehat{w}_{k}\|_{_{\mathbb{B}_{\eta}^{q}(\Gamma)}}^{q}  \leq 2\mathcal{H}    |\|(1,1)\||_{_{\widetilde{r},\widetilde{l},\Omega_{k},\Gamma_{k}}}^{p}|\|(f_{1}-f_{2},g_{1}-g_{2})\||_{_{r,l}}^{p}.
\end{equation}
Note that
\begin{equation}\label{1.13}
    \mu \left (\Gamma_{k} \right )^{^{\frac{p}{\widetilde{l}}}}\leq\mu \left (\Gamma \right )^{^{\frac{p-q}{\widetilde{l}}}}\mu \left (\Gamma_{k} \right )^{^{\frac{q}{\widetilde{l}}}}:=C_{\Gamma}\mu \left (\Gamma_{k} \right )^{^{\frac{q}{\widetilde{l}}}},
\end{equation}
so rewriting \eqref{1.12} and substituting \eqref{1.13}, we get that\\[2ex]
$\|\widehat{w}_{k}\|_{_{W^{s,p}(\Omega)}}^{p} +\|\widehat{w}_{k}\|_{_{\mathbb{B}_{\eta}^{q}(\Gamma)}}^{q}\leq 2\mathcal{H}\left ( |\Omega_{k}|^{\frac{1}{\widetilde{r}}} + \mu \left (\Gamma_{k} \right )^{\frac{1}{\widetilde{l}}}\right )^{p}|\|(f_{1}-f_{2},g_{1}-g_{2})\||_{_{r,l}}^{p}$\\
\begin{equation}\label{1.14}
\leq2^{p}\mathcal{H}\max\{1,C_{\Gamma}\}\left ( |\Omega_{k}|^{^{\frac{p}{\widetilde{r}}}} + \mu \left (\Gamma_{k} \right )^{^{\frac{q}{\widetilde{l}}}}\right )|\|(f_{1}-f_{2},g_{1}-g_{2})\||_{_{r,l}}^{p}.
\end{equation}
Applying Theorem \ref{sobolev-embedding} and Theorem \ref{besov-embedding} on the left side of \eqref{1.14} we have that
 \begin{equation}\label{1.17}
    \|\widehat{w}_{k}\|_{_{p^{*},\Omega_{k}}}^{p} +\|\widehat{w}_{k}\|_{_{q^{*},\Gamma_{k}}}^{q} \leq \widetilde{G}\left ( |\Omega_{k}|^{\frac{p}{\widetilde{r}}} + \mu \left (\Gamma_{k} \right )^{\frac{q}{\widetilde{l}}}\right )|\|(f_{1}-f_{2},g_{1}-g_{2})\||_{_{r,l}}^{p} ,
    \end{equation}
where 
\begin{equation*}
    \widetilde{G}=\widetilde{G}_{p,q}(f_{1},f_{2},g_{1},g_{2}):= 2^{p}\mathcal{H}\max\{1,C_{\Gamma}\}\max\{c_{1}^{p},c_{2}^{q}\}>0.
\end{equation*}
Next, given $h>k$, and note that $A_{h} \subset A_{k}$ and  $|\widehat{w}_{k}|\geq h-k$ over $A_{h}$ (refer to \eqref{1.27} for the definition of the $A_{k}$ ). With this in mind, we can conclude that
\begin{equation}\label{1.15}
    \|\widehat{w}_{k}\|_{_{p^{*},\Omega_{k}}}^{p} \geq (h-k)^{p}|\Omega_{h}|^{^{\frac{p}{p^{*}}}}
\indent\textrm{and}\indent
    \|\widehat{w}_{k}\|_{_{q^{*},\Gamma_{k}}}^{q} \geq (h-k)^{q}\mu \left (\Gamma_{h} \right )^{^{\frac{q}{q^{*}}}}.
\end{equation}
Thus, combining \eqref{1.15} and \eqref{1.17}, we deduce that
\begin{equation}\label{1.20}
    |\Omega_{h}|^{^{\frac{p}{p^{*}}}}+\mu \left (\Gamma_{h} \right )^{^{\frac{q}{q^{*}}}}\leq \widetilde{G} |\|(f_{1}-f_{2},g_{1}-g_{2})\||_{_{r,l}}^{p} \left ((h-k)^{-p} + (h-k)^{-q}  \right )\left ( |\Omega_{k}|^{^{\frac{p}{\widetilde{r}}}} + \mu \left (\Gamma_{k} \right )^{^{\frac{q}{\widetilde{l}}}}\right ).
\end{equation}
Setting 
\begin{equation}\label{1.24}
    \varPsi(t):=|\Omega_{t}|^{\frac{p}{p^{*}}}+\mu \left ( \Gamma_{t}\right )^{\frac{q}{q^{*}}}\indent\indent\indent\textrm{for each}\,\,t \in [0,\infty),
\end{equation}
from \eqref{1.20} it follows that 
\begin{equation}\label{1.21}
\varPsi(h)\leq\widetilde{G} |\|(f_{1}-f_{2},g_{1}-g_{2})\||_{_{r,l}}^{p} \left ((h-k)^{-p} + (h-k)^{-q}  \right )\left ( \varPsi(k)^{\frac{p^{*}}{\widetilde{r}}} + \varPsi(k)^{\frac{q^{*}}{\widetilde{l}}}\right ) .
\end{equation}
Since $r > \frac{Np}{Np+2sp-2N}$ and $s>\frac{dq}{dq+2\eta q-2d}$ for this case, notice that 
\begin{equation*}
    \delta:=\min \left \{\frac{p^{*}}{\widetilde{r}},\frac{q^{*}}{\widetilde{l}} \right \}>1,
\end{equation*}
and henceforth, it follows from \eqref{1.21} that 
\begin{equation}\label{1.21a}
    \varPsi(h) \leq \widetilde{G} |\|(f_{1}-f_{2},g_{1}-g_{2})\||_{_{r,l}}^{p} \left ((h-k)^{-p} + (h-k)^{-q}  \right )\left ( \varPsi(k)^{\frac{p^{*}}{\widetilde{r}}-\delta} + \varPsi(k)^{\frac{q^{*}}{\widetilde{l}}-\delta}\right  ) \varPsi(k)^{\delta} .
\end{equation}
Clearly $\varPsi(k)^{\frac{p^{*}}{\widetilde{r}}-\delta} + \varPsi(k)^{\frac{q^{*}}{\widetilde{l}}-\delta} \leq \mathcal{C}_{0}$ for some constant $\mathcal{C}_0>0$, and
inequality \eqref{1.21a} results into
\begin{equation*}
     \varPsi(h) \leq \mathcal{C}_{0} \widetilde{G} |\|(f_{1}-f_{2},g_{1}-g_{2})\||_{_{r,l}}^{p} \left ((h-k)^{-p} + (h-k)^{-q}  \right ) \varPsi(k)^{\delta} .
\end{equation*}
Applying  Lemma \ref{est} to the function $\varPsi(\cdot)$ for
\begin{equation*}
    k_{0}:=0, \qquad \alpha_{1}:=p, \qquad \alpha_{2}:=q, \qquad \text{and} \qquad c:=\mathcal{C}_{0} \widetilde{G} |\|(f_{1}-f_{2},g_{1}-g_{2})\||_{_{r,l}}^{p},
\end{equation*}
we get that 
\begin{equation}\label{1.22}
    \varPsi(\varsigma_{1}+\varsigma_{2})=0, \qquad \text{for} \qquad \varsigma_{1}^{p}=\varsigma_{2}^{q}:=C \ \widetilde{G} |\|(f_{1}-f_{2},g_{1}-g_{2})\||_{_{r,l}}^{p} \, ,
\end{equation}
for some constant $C=C(|\Omega|,\mu(\Gamma))>0$. Equation \eqref{1.22} and the definition of $\varPsi(\cdot)$ imply  that 
\begin{equation*}
    |\Omega_{\varsigma_{1}+\varsigma_{2}}|=\mu \left (\Gamma_{\varsigma_{1}+\varsigma_{2}} \right )=0,
\end{equation*}
and thus
\begin{equation}\label{1.23}
    |w|=|u_{1}-u_{2}|\leq\varsigma_{1}+\varsigma_{2}, \qquad \text{almost everywhere over $\overline{\Omega}$}.
\end{equation}
From here, combining \eqref{1.22} and \eqref{1.23}, we have that
\begin{eqnarray}
    |u_{1}-u_{2}|&\leq &\left (C \ \widetilde{G} |\|(f_{1}-f_{2},g_{1}-g_{2})\||_{_{r,l}}^{p} \right)^{\frac{1}{p}}+\left (C \ \widetilde{G} |\|(f_{1}-f_{2},g_{1}-g_{2})\||_{_{r,l}}^{p} \right )^{^{\frac{1}{q}}} \nonumber \\ 
    &\leq& C^{*} \left (  |\|(f_{1}-f_{2},g_{1}-g_{2})\||_{_{r,l}}  +  |\|(f_{1}-f_{2},g_{1}-g_{2})\||_{_{r,l}}^{^{\frac{p}{q}}} \right )\indent\textrm{a.e. in}\,\,\overline{\Omega},
\end{eqnarray}
where $C^{*}=C^{*}_{p,q}(f_{1},f_{2},g_{1},g_{2},|\Omega|,\mu(\Gamma)):=\max \{C^{\frac{1}{p}} \ \widetilde{G}^{\frac{1}{p}}, C^{\frac{1}{q}} \ \widetilde{G}^{\frac{1}{q}} \}. $
The $L^{\infty}$-bound follows for this case when $p>q$. Moreover, following the same procedure as before but switching the roles of $p$ and $q$, if $p<q$, then one obtains that
\begin{equation*}
    |u_{1}-u_{2}|<  C^{\blacktriangle} \left (  |\|(f_{1}-f_{2},g_{1}-g_{2})\||_{_{r,l}}  +  |\|(f_{1}-f_{2},g_{1}-g_{2})\||_{_{r,l}}^{^{\frac{q}{p}}} \right )\indent\textrm{a.e. in}\,\,\overline{\Omega}, 
\end{equation*}
for some constant $C^{\blacktriangle}>0$, and if we set $p=q$, there exists a constant $C^{\blacktriangledown}>0$ such that
\begin{equation*}
    |u_{1}-u_{2}|<  C^{\blacktriangledown} \left (  |\|(f_{1}-f_{2},g_{1}-g_{2})\||_{_{r,l}}  +  |\|(f_{1}-f_{2},g_{1}-g_{2})\||_{_{r,l}} \right )\indent\textrm{a.e. in}\,\,\overline{\Omega}.
\end{equation*}
Finally, combining all possibilities in this case, we can conclude that there exists a constant $$\mathscr{C}^{\blacktriangledown}=\mathscr{C}^{\blacktriangledown}(p,q, f_{1},f_{2},g_{1},g_{2},|\Omega|,\mu(\Gamma))>0$$ such that
\begin{equation*}
    |u_{1}-u_{2}|<  \mathscr{C}^{\blacktriangledown} \left (  |\|(f_{1}-f_{2},g_{1}-g_{2})\||_{_{r,l}}  +  |\|(f_{1}-f_{2},g_{1}-g_{2})\||_{_{r,l}}^{^{\frac{\max \{p,q\}}{\min \{ p,q\}}}}\indent\textrm{a.e. in}\,\,\overline{\Omega} \right ), 
\end{equation*}
establishing our desired result for the case $\frac{2N}{N+2s} < p <2$ and $\frac{2d}{d+2\eta} < q <2$. \\

\noindent$\bullet$ \underline{\textbf{Case 2}}: Assume that $p \geq 2$ and $q \geq 2$. In this case, the proof proceeds analogously to the previous case and an even simpler one. Given $\widehat{w}_{k}$ defined as in the previous case, a direct application of Lemma \ref{desigualdades_estimadas}(e,f) gives that
\begin{equation} \label{1.31}
    \varsigma_{0} \  \mathcal{E}_{p,q}(\widehat{w}_{k}, \widehat{w}_{k}) \leq \mathcal{E}_{p,q}(u_{1}, \widehat{w}_{k})-\mathcal{E}_{p,q}(u_{2}, \widehat{w}_{k}). 
\end{equation}
Proceeding as before, we calculate and arrive at
\begin{equation}\label{1.34}
    \|\widehat{w}_{k}\|_{_{p^{*},\Omega_{k}}}^{p} +\|\widehat{w}_{k}\|_{_{q^{*},\Gamma_{k}}}^{q} \leq \frac{\max\{c_{1}^{p},c_{2}^{q} \}}{M_{0} \varsigma_{0}}\left[ \mathcal{E}_{p,q}(u_{1}, \widehat{w}_{k})-\mathcal{E}_{p,q}(u_{2}, \widehat{w}_{k}) \right ].
\end{equation}
As $u_{1},u_{2} \in \mathbb{W}_{p,q}(\Omega,\Gamma)$ solve \eqref{ep} related to $(f_{1},g_{1})$ and $(f_{2},g_{2})$, respectively, testing \eqref{ws} with $\widehat{w}_{k}$ and applying H\"older inequality, one gets that
\begin{equation}\label{F05}
\mathcal{E}_{p,q}(u_{1}, \widehat{w}_{k})-\mathcal{E}_{p,q}(u_{2}, \widehat{w}_{k})\leq \|1\|_{_{\widetilde{r},\Omega_{k}}} \|f_{1}-f_{2}\|_{_{r,\Omega}} \|\widehat{w}_{k}\|_{_{p^{*},\Omega_{k}}} +\|1\|_{_{\widetilde{l},\Gamma_{k}}} \|g_{1}-g_{2}\|_{_{l,\Gamma}} \|\widehat{w}_{k}\|_{_{q^{*},\Gamma_{k}}} \, ,
\end{equation}
where $\widetilde{r},\widetilde{l} \in (1, \infty)$ are such that $\frac{1}{\widetilde{r}}+ \frac{1}{r}+\frac{1}{p^{*}}=1$ and $\frac{1}{\widetilde{l}}+ \frac{1}{l}+\frac{1}{q^{*}}=1$. Then, combining \eqref{1.34} with \eqref{F05} and applying Young's inequality, yields that\\[2ex]
$\|\widehat{w}_{k}\|_{_{p^{*},\Omega_{k}}}^{p} +\|\widehat{w}_{k}\|_{_{q^{*},\Gamma_{k}}}^{q}\leq \frac{\max\{c_{1}^{p},c_{2}^{q} \}}{M_{0} \varsigma_{0}}\|1\|_{_{\widetilde{r},\Omega_{k}}} \|f_{1}-f_{2}\|_{_{r,\Omega}} \|\widehat{w}_{k}\|_{_{p^{*},\Omega_{k}}} +   \frac{\max\{c_{1}^{p},c_{2}^{q} \}}{M_{0} \varsigma_{0}}\|1\|_{_{\widetilde{l},\Gamma_{k}}} \|g_{1}-g_{2}\|_{_{l,\Gamma}} \|\widehat{w}_{k}\|_{_{q^{*},\Gamma_{k}}}$ \\
$$\indent\indent\indent\indent\indent\leq \mathcal{C}\|1\|_{_{\widetilde{r},\Omega_{k}}}^{p'} \|f_{1}-f_{2}\|_{_{r,\Omega}}^{p'}  + \mathcal{C}\|1\|_{_{\widetilde{l},\Gamma_{k}}}^{q'} \|g_{1}-g_{2}\|_{_{l,\Gamma}}^{q'}  +  \epsilon \left (\|\widehat{w}_{k}\|_{_{p^{*},\Omega_{k}}}^{p} + \|\widehat{w}_{k}\|_{_{q^{*},\Gamma_{k}}}^{q} \right ),$$
for any $\epsilon \in (0,1)$, where
$$\mathcal{C}=\mathcal{C}_{p,q}:= \max \left \{ \left ( \frac{\max\{c_{1}^{p},c_{2}^{q} \}}{M_{0} \varsigma_{0} \epsilon^{1/p}} \right )^{p'},\left ( \frac{\max\{c_{1}^{p},c_{2}^{q} \}}{M_{0} \varsigma_{0} \epsilon^{1/q}} \right )^{q'} \right \}.$$
Selecting $\epsilon \in (0,1)$ suitably and setting $\mathcal{C}'=\mathcal{C}'_{p,q}:=\frac{\mathcal{C}}{1-\epsilon}$, we have
\begin{equation}\label{1.35}
    \|\widehat{w}_{k}\|_{_{p^{*},\Omega_{k}}}^{p} +\|\widehat{w}_{k}\|_{_{q^{*},\Gamma_{k}}}^{q} \leq \mathcal{C}'\left ( \|f_{1}-f_{2}\|_{_{r,\Omega}}^{p'} +\|g_{1}-g_{2}\|_{_{l,\Gamma}}^{q'} \right )\left ( 
    |\Omega_{k}|^{^{\frac{p'}{\widetilde{r}}}} + \mu \left ( \Gamma_{k} \right )^{^{\frac{q'}{\widetilde{l}}}}\right ).
\end{equation}
Taking $h>k$, we proceed similarly as in the proof of the first case to entail that
\begin{equation}\label{1.38}
|\Omega_{h}|^{^{\frac{p}{p^{*}}}} + \mu \left (\Gamma_{h} \right )^{^{\frac{q}{q^{*}}}}
     \leq \mathcal{C}'  \left ( \|f_{1}-f_{2}\|_{_{r,\Omega}}^{p'} +\|g_{1}-g_{2}\|_{_{l,\Gamma}}^{q'} \right ) \left ( (h-k)^{-p} + (h-k)^{-q} \right )\left ( 
    |\Omega_{k}|^{^{\frac{p'}{\widetilde{r}}}} + \mu \left ( \Gamma_{k} \right )^{^{\frac{q'}{\widetilde{l}}}}\right ).
\end{equation}
Recalling the definition of Definition of $\varPsi(\cdot)$ given in \eqref{1.24}, equation \eqref{1.38} becomes \\
\begin{equation}\label{1.39}
\varPsi(h)\leq  \mathcal{C}'  \left ( \|f_{1}-f_{2}\|_{_{r,\Omega}}^{p'} +\|g_{1}-g_{2}\|_{_{l,\Gamma}}^{q'} \right ) \left ( (h-k)^{-p} + (h-k)^{-q} \right )\left (\varPsi(k)^{\frac{p^{*}}{\widetilde{r}(p-1)}} + \varPsi(k)^{\frac{q^{*}}{\widetilde{l}(q-1)}}\right ).
\end{equation}
Since $r> \frac{N}{sp}$ and $l>\frac{d}{\eta q}$ in this case, selecting 
$$\delta:= \min \left \{\frac{p^{*}}{\widetilde{r}(p-1)}, \frac{q^{*}}{\widetilde{l}(q-1)} \right \} >1, $$
it follows (similarly as in the previous case) that there exists a constant $\mathcal{C}'_{0}=\mathcal{C}'_{0}(|\Omega|,\mu(\Gamma))>0$ such that
\begin{equation}\label{1.39b}
    \varPsi(h) \leq\mathcal{C}'_{0}  \left ( \|f_{1}-f_{2}\|_{_{r,\Omega}}^{p'} +\|g_{1}-g_{2}\|_{_{l,\Gamma}}^{q'} \right ) \left ( (h-k)^{-p} + (h-k)^{-q} \right )\varPsi(k)^{\delta}.
\end{equation}
Applying  Lemma \ref{est} to the function $\varPsi(\cdot)$ for
\begin{equation*}
    k_{0}:=0, \qquad \alpha_{1}:=p, \qquad \alpha_{2}:=q, \qquad \text{and} \qquad c:=\mathcal{C}'_{0}  \left ( \|f_{1}-f_{2}\|_{_{r,\Omega}}^{p'} +\|g_{1}-g_{2}\|_{_{l,\Gamma}}^{q'} \right ),
\end{equation*}
we get that
\begin{equation}\label{1.40}
   \varPsi( \varsigma_{1}' + \varsigma_{2}' )=0, \qquad \text{for} \qquad (\varsigma_{1}')^{p}=(\varsigma_{2}')^{q}=C'\left ( \|f_{1}-f_{2}\|_{_{r,\Omega}}^{p'} +\|g_{1}-g_{2}\|_{_{l,\Gamma}}^{q'} \right ) ,
\end{equation}
for some constant $C'=C'(|\Omega|,\mu(\Gamma))>0$. Thus, using \eqref{1.40}, we arrive at
\begin{equation}
|u_1+u_2|\,\leq\,\mathscr{C}^{\star} \left [   \|f_{1}-f_{2}\|_{_{r,\Omega}}^{\frac{1}{p-1}}   +  \|f_{1}-f_{2}\|_{_{r,\Omega}}^{\frac{p}{q(p-1)}}+\|g_{1}-g_{2}\|_{_{l,\Gamma}}^{\frac{1}{q-1}}  +\|g_{1}-g_{2}\|_{_{l,\Gamma}}^{\frac{q}{p(q-1)}} \right ]\indent\indent\textrm{a.e. in}\,\,\overline{\Omega},
\end{equation}
where $\mathscr{C}^{\star}=\mathscr{C}^{\star}(p,q,|\Omega|,\mu(\Gamma)):= \max\left \{C'^{\frac{1}{p}},C'^{\frac{1}{q}}\right \}$. This establishes 
the result for the case  $p \geq 2$ and $q \geq 2$.\\ 

\noindent$\bullet$ \underline{\textbf{Case 3}}: Assume that $p \geq 2$ and $\frac{2d}{d+2\eta} < q <2$. In this case, we are going to use the same notations and selections as in the previous cases. Additionally, due to the hypotheses of $p$ and $q$, this case is a combination of the two previous cases. Then, in views of \eqref{1.17},  \eqref{1.35} and the fact that $p>q$, we have 
\begin{equation}\label{1.41}
\|\widehat{w}_{k}\|_{_{q^{*},\Gamma_{k}}}^{q} \leq \widetilde{G}|\|(f_{1}-f_{2},g_{1}-g_{2})\||_{_{r,l}}^{p} \left ( |\Omega_{k}|^{\frac{p}{\widetilde{r}}} + \mu \left (\Gamma_{k} \right )^{\frac{q}{\widetilde{l}}}\right ),
\end{equation}
 and 
\begin{equation}\label{1.42}
\|\widehat{w}_{k}\|_{_{p^{*},\Omega_{k}}}^{p} \leq \mathcal{C}'\left ( \|f_{1}-f_{2}\|_{_{r,\Omega}}^{p'} +\|g_{1}-g_{2}\|_{_{l,\Gamma}}^{q'} \right )\left ( 
    |\Omega_{k}|^{\frac{p'}{\widetilde{r}}} + \mu \left ( \Gamma_{k} \right )^{\frac{q'}{\widetilde{l}}}\right ),
\end{equation}
where $\widetilde{r},\widetilde{l} \in (1, \infty)$ are such that $\frac{1}{\widetilde{r}}+ \frac{1}{r}+\frac{1}{p^{*}}=1$ and $\frac{1}{\widetilde{l}}+ \frac{1}{s}+\frac{1}{q^{*}}=1$. Then, combining \eqref{1.41} and \eqref{1.42}, it follows that \\[2ex]
$\|\widehat{w}_{k}\|_{_{p^{*},\Omega_{k}}}^{p} + \|\widehat{w}_{k}\|_{_{q^{*},\Gamma_{k}}}^{q}$\\[2ex] 
$\leq \widetilde{G}|\|(f_{1}-f_{2},g_{1}-g_{2})\||_{_{r,l}}^{p} \left ( |\Omega_{k}|^{\frac{p}{\widetilde{r}}} + \mu \left (\Gamma_{k} \right )^{\frac{q}{\widetilde{l}}}\right )+ \mathcal{C}'\left ( \|f_{1}-f_{2}\|_{_{r,\Omega}}^{p'} +\|g_{1}-g_{2}\|_{_{l,\Gamma}}^{q'} \right )\left ( 
|\Omega_{k}|^{\frac{p'}{\widetilde{r}}} + \mu \left ( \Gamma_{k} \right )^{\frac{q'}{\widetilde{l}}}\right ) $\\
\begin{equation}\label{1.43}
\leq K \left ( \|f_{1}-f_{2}\|_{_{r,\Omega}}^{p'} +\|g_{1}-g_{2}\|_{_{l,\Gamma}}^{q'} + |\|(f_{1}-f_{2},g_{1}-g_{2})\||_{_{r,l}}^{p} \right )
\times \left ( |\Omega_{k}|^{\frac{p}{\widetilde{r}}} + \mu \left (\Gamma_{k} \right )^{\frac{q}{\widetilde{l} }}+|\Omega_{k}|^{\frac{p'}{\widetilde{r}}} + \mu \left ( \Gamma_{k} \right )^{\frac{q'}{\widetilde{l} }} \right ),
\end{equation}
where $K=K_{p,q}(f_{1},f_{2},g_{1},g_{2}):=\max \{\widetilde{G}, \mathcal{C}'\}$. Since $p \geq 2$ and $\frac{2d}{d+2\eta} < q <2$, we see that
\begin{equation*}
    \frac{p'}{\widetilde{r}}=\frac{p}{\widetilde{r}(p-1)}<\frac{p}{\widetilde{r}} \qquad \text{and} \qquad \frac{q}{\widetilde{l} }<\frac{q}{\widetilde{l} (q-1)}=\frac{q'}{\widetilde{l} },
\end{equation*}
and consequently, 
\begin{equation}\label{1.44}
    |\Omega_{k}|^{\frac{p}{\widetilde{r}}}=|\Omega_{k}|^{\frac{p}{\widetilde{r}}-\frac{p'}{\widetilde{r}}}|\Omega_{k}|^{\frac{p'}{\widetilde{r}}} \leq  |\Omega|^{\frac{p-p'}{\widetilde{r}}}|\Omega_{k}|^{\frac{p'}{\widetilde{r}}}:= C'_{\Omega} |\Omega_{k}|^{\frac{p'}{\widetilde{r}}} ,
\end{equation}
and 
\begin{equation}\label{1.45}
    \mu \left (  \Gamma_{k} \right )^{\frac{q'}{\widetilde{l} }}=\mu \left (  \Gamma_{k} \right )^{\frac{q'}{\widetilde{l} }-\frac{q}{\widetilde{l} }}\mu \left (  \Gamma_{k} \right )^{\frac{q}{\widetilde{l} }} \leq  \mu \left (  \Gamma \right )^{\frac{q'-q}{\widetilde{l} }}\mu \left (  \Gamma_{k} \right )^{\frac{q}{\widetilde{l} }}:= C'_{\Gamma} \mu \left (  \Gamma_{k} \right )^{\frac{q}{\widetilde{l} }}.
\end{equation}
Substituting \eqref{1.44} and \eqref{1.45} into \eqref{1.43} gives that\\[2ex]
$\|\widehat{w}_{k}\|_{_{p^{*},\Omega_{k}}}^{p} + \|\widehat{w}_{k}\|_{_{q^{*},\Gamma_{k}}}^{q} $\\[2ex]
$\leq K \left ( \|f_{1}-f_{2}\|_{_{r,\Omega}}^{p'} +\|g_{1}-g_{2}\|_{_{l,\Gamma}}^{q'} + |\|(f_{1}-f_{2},g_{1}-g_{2})\||_{_{r,l}}^{p} \right )\left ( C'_{\Omega} |\Omega_{k}|^{\frac{p'}{\widetilde{r}}} + \mu \left (\Gamma_{k} \right )^{\frac{q}{\widetilde{l} }}+|\Omega_{k}|^{\frac{p'}{\widetilde{r}}} +C'_{\Gamma} \mu \left (  \Gamma_{k} \right )^{\frac{q}{\widetilde{l} }} \right ) $ \\
\begin{equation}\label{1.46}
 \leq \widetilde{K} \left ( \|f_{1}-f_{2}\|_{_{r,\Omega}}^{p'} +\|g_{1}-g_{2}\|_{_{l,\Gamma}}^{q'} + |\|(f_{1}-f_{2},g_{1}-g_{2})\||_{_{r,l}}^{p} \right )   \left (  |\Omega_{k}|^{\frac{p'}{\widetilde{r}}} +  \mu \left (\Gamma_{k} \right )^{\frac{q}{\widetilde{l} }} \right ),
\end{equation}
where 
\begin{equation*}
    \widetilde{K}=\widetilde{K}_{p,q}(f_{1},f_{2},g_{1},g_{2}):=K \ \max \{  1+C'_{\Omega},1+C'_{\Gamma} \}.
\end{equation*}
Now, proceeding similarly as to the previous cases, if we take $h>k$, then calculating we find that\\[2ex]
$|\Omega_{h}|^{\frac{p}{p^{*}}} + \mu \left (\Gamma_{h} \right )^{\frac{q}{q^{*}}}$ \\[2ex]
$\leq  \widetilde{K}(h-k)^{-p} \left ( \|f_{1}-f_{2}\|_{_{r,\Omega}}^{p'} +\|g_{1}-g_{2}\|_{_{l,\Gamma}}^{q'} + |\|(f_{1}-f_{2},g_{1}-g_{2})\||_{_{r,l}}^{p} \right ) \left (  |\Omega_{k}|^{\frac{p'}{\widetilde{r}}} +  \mu \left (\Gamma_{k} \right )^{\frac{q}{\widetilde{l} }} \right ) $\\

$\hspace{15pt}+ \widetilde{K}(h-k)^{-q} \left ( \|f_{1}-f_{2}\|_{_{r,\Omega}}^{p'} +\|g_{1}-g_{2}\|_{_{l,\Gamma}}^{q'} + |\|(f_{1}-f_{2},g_{1}-g_{2})\||_{_{r,l}}^{p} \right )\left (  |\Omega_{k}|^{\frac{p'}{\widetilde{r}}} +  \mu \left (\Gamma_{k} \right )^{\frac{q}{\widetilde{l} }} \right )$ \\
\begin{equation}\label{1.49}
=\widetilde{K}\left ( \|f_{1}-f_{2}\|_{_{r,\Omega}}^{p'} +\|g_{1}-g_{2}\|_{_{l,\Gamma}}^{q'} + |\|(f_{1}-f_{2},g_{1}-g_{2})\||_{_{r,l}}^{p} \right )\left ((h-k)^{-p} + (h-k)^{-q} \right )
\left (  |\Omega_{k}|^{\frac{p'}{\widetilde{r}}} +  \mu \left (\Gamma_{k} \right )^{\frac{q}{\widetilde{l} }} \right ).
\end{equation}
Using again the function $\varPsi(\cdot)$ defined by \eqref{1.24} and following the same approach as in the previous cases, we deduce that\\[2ex]
$\varPsi(h)$\\
\begin{equation}\label{1.50}
\leq \widetilde{K}\left ( \|f_{1}-f_{2}\|_{_{r,\Omega}}^{p'} +\|g_{1}-g_{2}\|_{_{l,\Gamma}}^{q'} + |\|(f_{1}-f_{2},g_{1}-g_{2})\||_{_{r,l}}^{p} \right )\left ((h-k)^{-p} + (h-k)^{-q} \right ) \left (  \varPsi(k)^{\frac{p^{*}}{\widetilde{r}(p-1)}} +  \varPsi(k)^{\frac{q^{*}}{\widetilde{l} }} \right ). 
\end{equation}
As $r> \frac{N}{sp}$ and $l>\frac{dq}{dq+2\eta q-2d}$ in this case, putting 
\begin{equation*}
    \delta:= \min\left \{\frac{p^{*}}{\widetilde{r}(p-1)} ,\frac{q^{*}}{\widetilde{l} } \right \} >1,
\end{equation*}
we obtain that
\begin{equation}\label{F06}
     \varPsi(h) \leq \mathcal{C}''_{0} \widetilde{K}\left ( \|f_{1}-f_{2}\|_{_{r,\Omega}}^{p'} +\|g_{1}-g_{2}\|_{_{l,\Gamma}}^{q'} + |\|(f_{1}-f_{2},g_{1}-g_{2})\||_{_{r,l}}^{p} \right )
     \left ((h-k)^{-p} + (h-k)^{-q} \right ) \varPsi(k)^{\delta},
\end{equation}
for some constant $\mathcal{C}''_0>0$.
An application of  Lemma \ref{est} to the function $\varPsi(\cdot)$ for
\begin{equation*}
    k_{0}:=0, \qquad \alpha_{1}:=p, \qquad \alpha_{2}:=q, \qquad \text{and}  
\end{equation*}
\begin{equation*}
    c:=\mathcal{C}''_{0} \widetilde{K}\left ( \|f_{1}-f_{2}\|_{_{r,\Omega}}^{p'} +\|g_{1}-g_{2}\|_{_{l,\Gamma}}^{q'} + |\|(f_{1}-f_{2},g_{1}-g_{2})\||_{_{r,l}}^{p} \right ),
\end{equation*}
gives that
\begin{equation}\label{F07}
    \varPsi( \varsigma_{1}'' + \varsigma_{2}'' )=0\,\,\,\,\,\textrm{for}\,\,\,\,\,
    (\varsigma_{1}'')^{p}=(\varsigma_{2}'')^{q}=C'' \widetilde{K}\left ( \|f_{1}-f_{2}\|_{_{r,\Omega}}^{p'} +\|g_{1}-g_{2}\|_{_{l,\Gamma}}^{q'} + |\|(f_{1}-f_{2},g_{1}-g_{2})\||_{_{r,l}}^{p} \right ),
\end{equation}
for some constant $C''=C''(|\Omega|,\mu(\Gamma))>0$. Inequality \eqref{F07} implies that\\

$|u_{1}-u_{2}| \leq \mathscr{C}^{\blacktriangle} \left [ \|f_{1}-f_{2}\|_{_{r,\Omega}}^{\frac{1}{p-1}}+ \|f_{1}-f_{2}\|_{_{r,\Omega}}^{\frac{p}{q(p-1)}}+\|g_{1}-g_{2}\|_{_{l,\Gamma}}^{\frac{1}{q-1}} +\|g_{1}-g_{2}\|_{_{l,\Gamma}}^{\frac{q}{p(q-1)}}   \right.$\\
\begin{equation}\label{F08}
    \left.+ |\|(f_{1}-f_{2},g_{1}-g_{2})\||_{_{r,l}} + |\|(f_{1}-f_{2},g_{1}-g_{2})\||_{_{r,l}}^{\frac{\max \{p,q \}}{\min \{p,q \}}} \right ]\indent\indent\textrm{a.e. in}\,\,\overline{\Omega},
\end{equation}
where $$\mathscr{C}^{\blacktriangle}=\mathscr{C}^{\blacktriangle}(p,q,f_{1},f_{2},g_{1},g_{2},|\Omega|, \mu(\Gamma)):=\max \left \{ (C'' \widetilde{K})^{\frac{1}{p}}, (C'' \widetilde{K})^{\frac{1}{q}} \right \}.$$
Equation \eqref{F08}  establishes the desired result for the case  $p \geq 2$ and $\frac{2d}{d+2\eta} < q <2$. \\ 

\noindent$\bullet$ \underline{\textbf{Case 4}}:  Assume that $\frac{2N}{N+2s} < p <2$ and $q \geq 2$. In this final case, by employing \eqref{1.35}, and similarly to how \eqref{1.41} was obtained, considering that $p<q$, we can conclude that
\begin{equation}\label{1.50d}
\|\widehat{w}_{k}\|_{_{q^{*},\Gamma_{k}}}^{q}\leq \mathcal{C}'\left ( \|f_{1}-f_{2}\|_{_{r,\Omega}}^{p'} +\|g_{1}-g_{2}\|_{_{l,\Gamma}}^{q'} \right )\left ( 
    |\Omega_{k}|^{\frac{p'}{\widetilde{r}}} + \mu \left ( \Gamma_{k} \right )^{\frac{q'}{\widetilde{l}}}\right ),
\end{equation}
and 
\begin{equation}\label{1.50e}
\|\widehat{w}_{k}\|_{_{p^{*},\Omega_{k}}}^{p} \leq \widetilde{G}'|\|(f_{1}-f_{2},g_{1}-g_{2})\||_{_{r,l}}^{q} \left ( |\Omega_{k}|^{\frac{p}{\widetilde{r}}} + \mu \left (\Gamma_{k} \right )^{\frac{q}{\widetilde{l}}}\right ),
\end{equation}
where $\widetilde{r},\widetilde{l} \in (1, \infty)$ are such that $\frac{1}{\widetilde{r}}+ \frac{1}{r}+\frac{1}{p^{*}}=1$ and $\frac{1}{\widetilde{l}}+ \frac{1}{s}+\frac{1}{q^{*}}=1$. Taking into account \eqref{1.50d}, \eqref{1.50e}, and following a step-by-step approach similar to the previous case, we can conclude that there exists a constant 
$\mathscr{C}^{\diamond}=\mathscr{C}^{\diamond}(p,q,f_{1},f_{2},g_{1},g_{2},|\Omega|, \mu(\Gamma))>0$ such that\\[2ex]
$|u_{1}-u_{2}|\leq\mathscr{C}^{\diamond}  \left [ \|f_{1}-f_{2}\|_{_{r,\Omega}}^{\frac{1}{p-1}}+ \|f_{1}-f_{2}\|_{_{r,\Omega}}^{\frac{p}{q(p-1)}}+\|g_{1}-g_{2}\|_{_{l,\Gamma}}^{\frac{1}{q-1}} +\|g_{1}-g_{2}\|_{_{l,\Gamma}}^{\frac{q}{p(q-1)}}\right.$\\
\begin{equation*}
\indent\indent\indent\indent\indent\indent\indent\left.+ |\|(f_{1}-f_{2},g_{1}-g_{2})\||_{_{r,l}} + |\|(f_{1}-f_{2},g_{1}-g_{2})\||_{_{r,l}}^{\frac{\max \{p,q \}}{\min \{p,q \}}} \right ].
\end{equation*}
Therefore, by combining these four cases, we have proven the desired result.
\end{proof}

\section{Maximum principle}\label{sec5}

In this section we explore inverse positivity and a weak comparison principle for solutions of the boundary value problem \eqref{ep}, built upon the assumptions established in the previous sections.\\
\indent To begin, consider the differential inequality, formally given by  
\begin{equation}
\label{iep}\left\{
\begin{array}{lcl}
(-\Delta)^s_{_{p,\Omega}}u+\alpha|u|^{p-2}u\,\geq \,0\indent\indent\indent\,\,\,\,\textrm{in}\,\,\Omega,\\ \\
C_{p,s}\mathcal{N}^{p'(1-s)}_pu+\beta|u|^{q-2}u+\Theta^{\eta}_qu\,\geq \,0\indent\,\textrm{on}\,\,\Gamma,\\
\end{array}
\right.
\end{equation}
where $p,q,\alpha$ and $\beta$ are defined in the same way as in problem \eqref{ep}. 

\begin{definition}
    A function $u \in \mathbb{W}_{p,q}(\Omega,\Gamma)$ is called a {\bf weak solution of the differential inequality \eqref{iep}}, if 
  \begin{equation}\label{1.74}
        \mathcal{E}_{p,q}(u,\Phi) \geq 0, \qquad \forall \Phi \in \mathbb{W}_{p,q}^{+}(\Omega,\Gamma),
    \end{equation}
where
\begin{equation*}
    \mathbb{W}_{p,q}^{+}(\Omega,\Gamma):=\{ w \in \mathbb{W}_{p,q}(\Omega,\Gamma) \ | \ w(x) \geq 0 \ \ \text{for almost every } \  x \in \overline{\Omega}  \},
\end{equation*}
and the nonlinear form $\mathcal{E}_{p,q}(\cdot,\cdot)$ is defined by \eqref{form}.
\end{definition}

\indent Our first result shows that weak solutions to the differential inequality \eqref{iep} are globally nonnegative. We recall that for a measurable function $v$, we write
$v^{+}=\max\{v,0\}$ and $u^{-}=-\min\{-u,0\}$.

\begin{theorem}\label{1.75}
    Let $u\in \mathbb{W}_{p,q}(\Omega,\Gamma)$ be a weak solution of the differential inequality \eqref{iep}. Then $u(x) \geq 0 $ for almost every $x \in \overline{\Omega}$ (a.e. in $\Omega$ and $\mu$-a.e. on $\Gamma$).
\end{theorem}

\begin{proof}
    Let $u \in \mathbb{W}_{p,q}(\Omega,\Gamma)$ be a weak solution of \eqref{iep}. For $\mathcal{D}$ denoting either $\Omega$ or $\Gamma$, we put 
\begin{equation*}
    \mathcal{D}^{*}:=\{  x \in \mathcal{D} \ | \ u(x) \geq 0 \}  \ \ \text{and} \ \  \mathcal{D}_{*}:=\mathcal{D} \backslash \mathcal{D}^{*}.
\end{equation*}
Then, testing \eqref{1.74} with the function  $u^{-} \in  \mathbb{W}_{p,q}^{+}(\Omega,\Gamma)$, and using the fact that $u=u^{+}-u^{-}$, we get \\ 

$\mathcal{E}_{p,q}(u,u^{-})$\\

{\scriptsize $=C_{N,p,s} \left [ \displaystyle\int_{\Omega_{*}} \int_{\Omega_{*}} \frac{|(u^{+}-u^{-})(x)-(u^{+}-u^{-})(y)|^{p-2}((u^{+}-u^{-})(x)-(u^{+}-u^{-})(y))(u^{-}(x)-u^{-}(y))}{|x-y|^{N+sp}} \, \mathrm{d}x \, \mathrm{d}y  \right.$}\\

{\footnotesize $\hspace{15pt}+\displaystyle\int_{\Omega_{*}} \int_{\Omega^{*}}  \frac{|(u^{+}-u^{-})(x)-(u^{+}-u^{-})(y)|^{p-2}((u^{+}-u^{-})(x)-(u^{+}-u^{-})(y))(u^{-}(x))}{|x-y|^{N+sp}} \, \mathrm{d}x \, \mathrm{d}y  $ } \\

{\footnotesize $\hspace{15pt} \left. +\displaystyle\int_{\Omega^{*}} \int_{\Omega_{*}}  \frac{|(u^{+}-u^{-})(x)-(u^{+}-u^{-})(y)|^{p-2}((u^{+}-u^{-})(x)-(u^{+}-u^{-})(y))(-u^{-}(y))}{|x-y|^{N+sp}} \, \mathrm{d}x \, \mathrm{d}y \right ] $} \\

$\hspace{15pt} + \displaystyle\int_{\Omega_{*}}\alpha |(u^{+}-u^{-})|^{p-2}(u^{+}-u^{-})u^{-}  \, \mathrm{d}x$\\

{\footnotesize $\hspace{6pt} +\displaystyle\int_{\Gamma_{*}} \int_{\Gamma_{*}} \frac{|(u^{+}-u^{-})(x)-(u^{+}-u^{-})(y)|^{q-2}((u^{+}-u^{-})(x)-(u^{+}-u^{-})(y))(u^{-}(x)-u^{-}(y))}{|x-y|^{d+\eta q}} \, \mathrm{d}\mu_{x} \, \mathrm{d}\mu_{y}$} \\

{\footnotesize $\hspace{15pt}+\displaystyle\int_{\Gamma_{*}} \int_{\Gamma^{*}} \frac{|(u^{+}-u^{-})(x)-(u^{+}-u^{-})(y)|^{q-2}((u^{+}-u^{-})(x)-(u^{+}-u^{-})(y))(u^{-}(x))}{|x-y|^{d+\eta q}} \, \mathrm{d}\mu_{x} \, \mathrm{d}\mu_{y}$} \\

{\footnotesize $\hspace{15pt}+\displaystyle\int_{\Gamma^{*}} \int_{\Gamma_{*}} \frac{|(u^{+}-u^{-})(x)-(u^{+}-u^{-})(y)|^{q-2}((u^{+}-u^{-})(x)-(u^{+}-u^{-})(y))(-u^{-}(y))}{|x-y|^{d+\eta q}} \, \mathrm{d}\mu_{x} \, \mathrm{d}\mu_{y} $} \\ 

$\hspace{15pt}+ \displaystyle\int_{\Gamma_{*}}\beta |(u^{+}-u^{-})|^{q-2}(u^{+}-u^{-})u^{-} \, \mathrm{d}\mu_{x}$ \\

$=-C_{N,p,s}\left [ \displaystyle\int_{\Omega_{*}} \int_{\Omega_{*}} \frac{|u^{-}(x)-u^{-}(y)|^{p}}{|x-y|^{N+sp}}  \, \mathrm{d}x \, \mathrm{d}y \right. $ \\

$\hspace{15pt} + \displaystyle\int_{\Omega_{*}} \int_{\Omega^{*}}  \frac{|u^{-}(x)+u^{+}(y)|^{p-2}(u^{-}(x)+u^{+}(y))(u^{-}(x))}{|x-y|^{N+sp}}  \, \mathrm{d}x \, \mathrm{d}y $ \\

$\hspace{15pt}\left. + \displaystyle\int_{\Omega^{*}} \int_{\Omega_{*}}  \frac{|u^{+}(x)+u^{-}(y)|^{p-2}(u^{+}(x)+u^{-}(y))(u^{-}(y))}{|x-y|^{N+sp}}  \, \mathrm{d}x \, \mathrm{d}y \right ] - \int_{\Omega_{*}}\alpha |u^{-}|^{p}  \, \mathrm{d}x $ \\

$\hspace{15pt}-\displaystyle\int_{\Gamma_{*}} \int_{\Gamma_{*}} \frac{|u^{-}(x)-u^{-}(y)|^{q}}{|x-y|^{d+\eta q}} \, \mathrm{d}\mu_{x} \, \mathrm{d}\mu_{y}$ \\

$\hspace{15pt}-\displaystyle\int_{\Gamma_{*}} \int_{\Gamma^{*}} \frac{|u^{-}(x)+u^{+}(y)|^{q-2}(u^{-}(x)+u^{+}(y))(u^{-}(x))}{|x-y|^{d+\eta q}} \, \mathrm{d}\mu_{x} \, \mathrm{d}\mu_{y} $ \\

$\hspace{15pt}-\displaystyle\int_{\Gamma^{*}} \int_{\Gamma_{*}} \frac{|u^{+}(x)+u^{-}(y)|^{q-2}(u^{+}(x)+u^{-}(y))(u^{-}(y))}{|x-y|^{d+\eta q}} \, \mathrm{d}\mu_{x} \, \mathrm{d}\mu_{y}  - \int_{\Gamma_{*}}\beta |u^{-}|^{q} \, \mathrm{d}\mu_{x}$ \\ 

$\leq - \left \{ C_{N,p,s}\left [ \displaystyle\int_{\Omega_{*}} \int_{\Omega_{*}} \frac{|u^{-}(x)-u^{-}(y)|^{p}}{|x-y|^{N+sp}}  \, \mathrm{d}x \, \mathrm{d}y \right. \right.\\$
     
$ \hspace{15pt}\left.+ \displaystyle\int_{\Omega_{*}} \int_{\Omega^{*}}  \frac{|u^{-}(x)|^{p}}{|x-y|^{N+sp}}  \, \mathrm{d}x \, \mathrm{d}y  + \int_{\Omega^{*}} \int_{\Omega_{*}}  \frac{|u^{-}(y)|^{p}}{|x-y|^{N+sp}}  \, \mathrm{d}x \, \mathrm{d}y \right ] + \int_{\Omega_{*}}\alpha |u^{-}|^{p}  \, \mathrm{d}x $ \\

 $\hspace{15pt}+\displaystyle\int_{\Gamma_{*}} \int_{\Gamma_{*}} \frac{|u^{-}(x)-u^{-}(y)|^{q}}{|x-y|^{d+\eta q}} \, \mathrm{d}\mu_{x} \, \mathrm{d}\mu_{y} +\int_{\Gamma_{*}} \int_{\Gamma^{*}} \frac{|u^{-}(x)|^{q}}{|x-y|^{d+\eta q}} \, \mathrm{d}\mu_{x} \, \mathrm{d}\mu_{y} $ \\

$\hspace{15pt} \left.+\displaystyle\int_{\Gamma^{*}} \int_{\Gamma_{*}} \frac{|u^{-}(y)|^{q}}{|x-y|^{d+\eta q}} \, \mathrm{d}\mu_{x} \, \mathrm{d}\mu_{y}  + \int_{\Gamma_{*}}\beta |u^{-}|^{q} \, \mathrm{d}\mu_{x} \right \}$ \\
$$\indent\indent\indent\indent= -\mathcal{E}_{p,q}(u^{-},u^{-})\leq 0.$$
Since $u\in \mathbb{W}_{p,q}(\Omega,\Gamma)$ is a weak solution of the differential inequality \eqref{iep}, in views of above inequality, we have 
\begin{equation*}
    0\leq \mathcal{E}_{p,q} (u,u^{-}) \leq -\mathcal{E}_{p,q}(u^{-},u^{-})\leq 0,
\end{equation*}
and thus
\begin{equation}\label{DI01}
    \mathcal{E}_{p,q}(u^{-},u^{-})=0.
\end{equation}
Furthermore, by \eqref{1.56}, there exists a constant $M_0>0$ such that 
\begin{equation}\label{DI02}
    \mathcal{E}_{p,q}(u^{-},u^{-}) \geq M_{0} \left (||u^{-}||_{_{W^{s,p}(\Omega)}}^{p}+||u^{-}||_{_{\mathbb{B}_{\eta}^{q}(\Gamma)}}^{q} \right ) \geq 0.
\end{equation}
Combining \eqref{DI01} and \eqref{DI02}, we deduce that $(u^{-},u^{-}|_{\Gamma})=(0,0)$ a.e. on $\Omega \times \Gamma$, which implies that $u=u^{+} \geq 0$  a.e over $\overline{\Omega}$, as desired.
\end{proof}

\indent Taking into account the previous result, we can easily derive the following result on inverse positivity.

\begin{corollary}
    Let $(f,g) \in \mathbb{X}^{r,l}(\Omega,\Gamma)$, where $r \geq (p^{*})'$ and $l \geq (q^{*})'$, and let $u \in \mathbb{W}_{p,q}(\Omega,\Gamma)$ be a weak solution of the problem  \eqref{ep}. If $f\geq 0$ a.e in $\Omega$ and $g \geq 0$ $\mu-$a.e on $\Gamma$, then $u(x) \geq 0 $ for almost every $x \in \overline{\Omega}$ (a.e. in $\Omega$ and $\mu$-a.e. on $\Gamma$).
\end{corollary}

\indent To conclude this section, we introduce a fundamental tool for analyzing the behavior of weak solutions to problem \eqref{ep}, namely, a Weak Comparison Principle.

\begin{theorem}\label{WCP}
    Given $(f_{1},g_{1}),(f_{2},g_{2})  \in \mathbb{X}^{r,l}(\Omega,\Gamma)$, for $r \geq (p^{*})'$ and $l \geq (q^{*})'$, let $u_{1},u_{2} \in \mathbb{W}_{p,q}(\Omega,\Gamma)$ be weak solutions of problem \eqref{ep} related to $(f_{1},g_{1})$ and $(f_{2},g_{2})$, respectively. If $f_{1} \geq f_{2}$ a.e in $\Omega$ and $g_{1} \geq g_{2}$ $\mu-$a.e on $\Gamma$, then $u_{1}(x) \geq u_{2}(x)$ for almost every $x \in \overline{\Omega}$ (a.e. in $\Omega$ and $\mu$-a.e. on $\Gamma$).
\end{theorem}

\begin{proof}
    Given $u_{1}$ and $u_{2}$ as in the Theorem, since $f_{1} \geq f_{2}$ a.e in $\Omega$ and $g_{1} \geq g_{2}$ $\mu-$a.e on $\Gamma$, then for each $w \in \mathbb{W}_{p,q}^{+}(\Omega,\Gamma)$ we have
    \begin{equation*}
        \mathcal{E}_{p,q}(u_{1},w)=\int
    _{\Omega} f_{1} w \, \mathrm{d}x + \int_{\Gamma} g_{1} w \, \mathrm{d}\mu_{x} \geq \int
    _{\Omega} f_{2} w \, \mathrm{d}x + \int_{\Gamma} g_{2} w \, \mathrm{d}\mu_{x}=\mathcal{E}_{p,q}(u_{2},w),
    \end{equation*} 
and thus
$$0\leq \mathcal{E}_{p,q}(u_{1},w)-\mathcal{E}_{p,q}(u_{2},w).$$
On the other hand, clearly\\[2ex]
$\mathcal{E}_{p,q}(u_{1},w)-\mathcal{E}_{p,q}(u_{2},w)$ \\

\resizebox{\textwidth}{!}{$=C_{N,p,s} \displaystyle\int_{\Omega}\int_{\Omega}\frac{\left [|u_{1}(x)-u_{1}(y)|^{p-2}(u_{1}(x)-u_{1}(y))-|u_{2}(x)-u_{2}(y)|^{p-2}(u_{2}(x)-u_{2}(y))\right ](w(x)-w(y))}{|x-y|^{N+sp}} \, \mathrm{d}x \, \mathrm{d}y$} \\ 

\resizebox{\textwidth}{!}{$\hspace{12pt}+\displaystyle\int_{\Gamma}\int_{\Gamma}\frac{\left [|u_{1}(x)-u_{1}(y)|^{q-2}(u_{1}(x)-u_{1}(y))-|u_{2}(x)-u_{2}(y)|^{q-2}(u_{2}(x)-u_{2}(y))\right ](w(x)-w(y))}{|x-y|^{d+\eta q}} \, \mathrm{d}\mu_{x} \, \mathrm{d}\mu_{y}$} \\

$\hspace{12pt}+ \displaystyle\int_{\Omega}\alpha\left [|u_{1}|^{p-2}u_{1}- |u_{2}|^{p-2}u_{2}\right ]w \, \mathrm{d}x +  \int_{\Gamma}\beta\left [|u_{1}|^{q-2}u_{1}- |u_{2}|^{q-2}u_{2}\right ]w \, \mathrm{d}\mu_{x}$. \\

\noindent In particular, choosing $w=(u_{2}-u_{1})^{+}$ in the above inequality and taking 
$$b=b(x,y)=u_{1}(x)-u_{1}(y) \quad \text{and} \quad a=a(x,y)=u_{2}(x)-u_{2}(y),$$
it follows that \\[2ex]
$0 \leq C_{N,p,s} \displaystyle\int_{\Omega}\int_{\Omega}\frac{\left [|b|^{p-2}b-|a|^{p-2}a\right ]((u_{2}-u_{1})^{+}(x)-(u_{2}-u_{1})^{+}(y))}{|x-y|^{N+sp}} \, \mathrm{d}x \, \mathrm{d}y$  \\

$\hspace{20pt}+\displaystyle\int_{\Gamma}\int_{\Gamma}\frac{\left [|b|^{q-2}b-|a|^{q-2}a\right ]((u_{2}-u_{1})^{+}(x)-(u_{2}-u_{1})^{+}(y))}{|x-y|^{d+\eta q}} \, \mathrm{d}\mu_{x} \, \mathrm{d}\mu_{y} $

\begin{equation}\label{1.77}
+ \displaystyle\int_{\Omega}\alpha\left [|u_{1}|^{p-2}u_{1}- |u_{2}|^{p-2}u_{2}\right ](u_{2}-u_{1})^{+} \, \mathrm{d}x
+  \int_{\Gamma}\beta\left [|u_{1}|^{q-2}u_{1}- |u_{2}|^{q-2}u_{2}\right ](u_{2}-u_{1})^{+} \, \mathrm{d}\mu_{x}
\end{equation}
Next, in either $\Omega$ or $\Gamma$, writing $\varPsi:=u_{2}-u_{1}$, we calculate and deduce that
$$(\varPsi(y)-\varPsi(x))(\varPsi^{+}(x)-\varPsi^{+}(y))=-\left\{[\varPsi^{+}(x)-\varPsi^{+}(y)]^{2}+\varPsi^{-}(y)\varPsi^{+}(x)+\varPsi^{-}(x)\varPsi^{+}(y)\right\}\leq0.$$
Taking into account this latest calculation, the definitions of $a$, \,$b$, \,$\varPsi$, and the equality
\begin{equation}\label{1.78}
    |b|^{p-2}b-|a|^{p-2}a=(p-1)(b-a)Q(x,y)
\indent\,\,\textrm{for}\,\,\indent
    Q(x,y)=\int_{0}^{1} |a + t(b-a)|^{p-2} \, dt \geq 0,
\end{equation}
we get that\\[2ex]
$\displaystyle\frac{\left [|b|^{p-2}b-|a|^{p-2}a\right ]((u_{2}-u_{1})^{+}(x)-(u_{2}-u_{1})^{+}(y))}{|x-y|^{N+sp}}$ \\

$=\displaystyle\frac{(p-1)Q(x,y)}{|x-y|^{N+sp}}(b-a)((u_{2}-u_{1})^{+}(x)-(u_{2}-u_{1})^{+}(y))$ \\

$=-\displaystyle\frac{(p-1)Q(x,y)}{|x-y|^{N+sp}}\left\{[\varPsi^{+}(x)-\varPsi^{+}(y)]^{2}+\varPsi^{-}(y)\varPsi^{+}(x)+\varPsi^{-}(x)\varPsi^{+}(y) \right \}\leq 0.$\\[2ex]
Consequently, 
\begin{equation}\label{1.80}
    C_{N,p,s} \int_{\Omega}\int_{\Omega}\frac{\left [|b|^{p-2}b-|a|^{p-2}a\right ]((u_{2}-u_{1})^{+}(x)-(u_{2}-u_{1})^{+}(y))}{|x-y|^{N+sp}} \, \mathrm{d}x \, \mathrm{d}y \leq 0,
\end{equation}
and similarly,
\begin{equation}\label{1.81}
    \int_{\Gamma}\int_{\Gamma}\frac{\left [|b|^{q-2}b-|a|^{q-2}a\right ]((u_{2}-u_{1})^{+}(x)-(u_{2}-u_{1})^{+}(y))}{|x-y|^{d+\eta q}} \, \mathrm{d}\mu_{x} \, \mathrm{d}\mu_{y} \leq 0.
\end{equation}
Now, for $\mathcal{D}$ denoting either $\Omega$ or $\Gamma$, we put 
\begin{equation*}
    \mathcal{D}^{+}:=\{ x \in \mathcal{D} \ | \ u_{2} \geq u_{1} \} \qquad  \text{and} \qquad  \mathcal{D}^{-}:=\{ x \in \mathcal{D} \ | \ u_{2} < u_{1} \}.
\end{equation*}
Clearly $(u_{2}-u_{1})^{+}=u_{2}-u_{1}$ in $\mathcal{D}^{+}$ and $(u_{2}-u_{1})^{+}=0$ in $\mathcal{D}^{-}$, and an application of \eqref{r>1} gives\\
\begin{equation}\label{1.82}
\int_{\Omega}\alpha\left [|u_{1}|^{p-2}u_{1}- |u_{2}|^{p-2}u_{2}\right ](u_{2}-u_{1})^{+} \, \mathrm{d}x=-\int_{\Omega^{+}}\alpha\left [|u_{1}|^{p-2}u_{1}- |u_{2}|^{p-2}u_{2}\right ](u_{1}-u_{2}) \, \mathrm{d}x\leq0,
\end{equation}
and similarly,
\begin{equation}\label{1.83}
    \int_{\Gamma}\beta\left [|u_{1}|^{q-2}u_{1}- |u_{2}|^{q-2}u_{2}\right ](u_{2}-u_{1})^{+} \, \mathrm{d}\mu_{x} \leq 0.
\end{equation}
From \eqref{1.77} together with \eqref{1.80}, \eqref{1.81}, \eqref{1.82} and \eqref{1.83}, we conclude that \\[2ex]
$C_{N,p,s} \displaystyle\int_{\Omega}\int_{\Omega}\frac{\left [|b|^{p-2}b-|a|^{p-2}a\right ]((u_{2}-u_{1})^{+}(x)-(u_{2}-u_{1})^{+}(y))}{|x-y|^{N+sp}} \, \mathrm{d}x \, \mathrm{d}y$  \\

$\hspace{20pt}+\displaystyle\int_{\Gamma}\int_{\Gamma}\frac{\left [|b|^{q-2}b-|a|^{q-2}a\right ]((u_{2}-u_{1})^{+}(x)-(u_{2}-u_{1})^{+}(y))}{|x-y|^{d+\eta q}} \, \mathrm{d}\mu_{x} \, \mathrm{d}\mu_{y} $

\begin{equation}\label{1.84}
+ \displaystyle\int_{\Omega}\alpha\left [|u_{1}|^{p-2}u_{1}- |u_{2}|^{p-2}u_{2}\right ](u_{2}-u_{1})^{+} \, \mathrm{d}x
+  \int_{\Gamma}\beta\left [|u_{1}|^{q-2}u_{1}- |u_{2}|^{q-2}u_{2}\right ](u_{2}-u_{1})^{+} \, \mathrm{d}\mu_{x}=0.
\end{equation}
Since each term in  \eqref{1.84} has the same sign, this implies that
 $$\int_{\Omega}\alpha\left [|u_{1}|^{p-2}u_{1}- |u_{2}|^{p-2}u_{2}\right ](u_{2}-u_{1})^{+} \, \mathrm{d}x=0$$
and 
$$\int_{\Gamma}\beta\left [|u_{1}|^{q-2}u_{1}- |u_{2}|^{q-2}u_{2}\right ](u_{2}-u_{1})^{+} \, \mathrm{d}\mu_{x}=0.$$
In the last two integrals, each term within them has a defined sign, so this implies that $(u_{2}-u_{1})^{+}=0$ a.e in $\Omega$, and $(u_{2}-u_{1})^{+}=0$ a.e on $\Gamma$, which also says that $|\Omega^{+}|=\mu(\Gamma^{+})=0$. Consequently, one has that $u_{2}-u_{1} \leq 0$ a.e. in $\Omega$ and  $u_{2}-u_{1} \leq 0$ a.e. on $\Gamma$. Therefore,  $u_{1}(x) \geq u_{2}(x)$ for almost every $x \in \overline{\Omega}$, which completes the proof.
\end{proof}

\section{Nonlinear Fredholm alternative}\label{sec6}

In this section, we establish a sort of ``nonlinear Fredholm alternative'' for our problem, assuming the conditions and structures given in \eqref{A}. To begin, we define the sets  
\begin{equation*}
    W^{s,p}_{0}(\Omega):=\left \{ u \in W^{s,p}(\Omega)  \mid \int_{\Omega} u \, \mathrm{d}x=0 \right \}\indent\textrm{and}\indent
   \mathbb{B}^{q}_{\eta,0}(\Gamma):=\left \{ u \in \mathbb{B}^{q}_{\eta}(\Gamma)  \mid \int_{\Gamma} u \, \mathrm{d}\mu=0 \right \}. 
\end{equation*}
It follows that $ W^{s,p}_{0}(\Omega)$ and $\mathbb{B}^{q}_{\eta,0}(\Gamma)$ are closed subspaces of $ W^{s,p}(\Omega)$ and $\mathbb{B}^{q}_{\eta}(\Gamma)$, respectively, and consequently, they are Banach spaces with respect to the norms  
\begin{equation*}
    \|\cdot \|_{_{W^{s,p}_{0}(\Omega)}}:=\|\cdot \|_{_{W^{s,p}(\Omega)}} \qquad \text{and} \qquad \|\cdot \|_{_{\mathbb{B}^{q}_{\eta,0}(\Gamma)}}:=\|\cdot \|_{_{\mathbb{B}^{q}_{\eta}(\Gamma)}}.
\end{equation*}
The following simple result shows that the fractional and Besov semi-norms become norms when restricted to the spaces $W^{s,p}_{0}(\Omega)$ and $\mathbb{B}^{q}_{\eta,0}(\Gamma)$, respectively.

\begin{proposition}\label{equi-norms}
The norms
    \begin{equation*}
    \|u\|_{_{\mathscr{W}^{s,p}_{0}(\Omega)}}:=\left [\mathfrak{N}_{_{\Omega}}^{p}(u) \right ]^{1/p} \qquad \text{and} \qquad \|u\|_{_{\mathscr{B}^{q}_{\eta,0}(\Gamma)}}:= \left [ \mathfrak{N}_{_{\Gamma}}^{q}(u) \right ]^{1/q}
\end{equation*}
define equivalent norms for the spaces $W^{s,p}_{0}(\Omega)$ and $\mathbb{B}^{q}_{\eta,0}(\Gamma)$, respectively, where
\begin{equation*}
   \mathfrak{N}_{_{\Omega}}^{p}(u):=\int_{\Omega}\int_{\Omega}\frac{|u(x)-u(y)|^{p}}
{|x-y|^{N+sp}}\,\mathrm{d}x \, \mathrm{d}y \qquad \text{and} \qquad  \mathfrak{N}_{_{\Gamma}}^{q}(u):=\int_{\Gamma}\int_{\Gamma}\frac{|u(x)-u(y)|^{q}}
{|x-y|^{\eta q+d}}\,\mathrm{d}\mu_{x} \, \mathrm{d}\mu_{y}.
\end{equation*}
\end{proposition}

\begin{proof}  
If $u \in W^{s,p}_{0}(\Omega) $, then an application of Poincar\'e’s inequality (e.g. \cite[Theorem 6.33]{leoni2023first}) together with the fact that $ u $ integrates to zero over $ \Omega $ gives that  
\begin{equation*}  
     \|u \|_{_{W^{s,p}_{0}(\Omega)}}^{p}=\left\|u-u_{_{\Omega}}\right\|_{_{W^{s,p}(\Omega)}}^{p}=\left\|u-u_{_{\Omega}}\right\|_{_{p,\Omega}}^{p} +\mathfrak{N}_{_{\Omega}}^{p}(u-u_{_{\Omega}}) \leq C \, \mathfrak{N}_{_{\Omega}}^{p}(u),  
\end{equation*}  
where $ u_{_{\Omega}}:=\frac{1}{|\Omega|} \int_{\Omega} u \, \mathrm{d} x$ denotes the average of $u$ over $ \Omega$. Obviously, $\mathfrak{N}_{_{\Omega}}^{p}(u)\leq \|u \|_{_{W^{s,p}(\Omega)}}^{p}=\|u \|_{_{W^{s,p}_{0}(\Omega)}}^{p}$ for each $u\in W^{s,p}_{0}(\Omega)$, and thus, the combination of these inequalities gives the desired equivalence of norms for $W^{s,p}_{0}(\Omega)$. Turning our attention to the boundary Besov semi-norm, 
using the Poincar\'e-type inequality in \cite[Theorem 11.1]{danielli2006non} together with the fact $C^{\infty}(\overline{\Omega})$ is dense in the fractional $W^{s,q}$-type Sobolev spaces over $\Omega$ (in the same spirit as in \cite[section 10]{danielli2006non}), we get $\left\|u-u_{_{\Gamma}}\right\|^q_{_{q,\Gamma}}\leq C_{_{\Gamma}}\mathfrak{N}_{_{\Gamma}}^{q}(u)$ for some constant $C_{_{\Gamma}}>0$, where $u_{_{\Gamma}}:=\frac{1}{\mu(\Gamma)} \int_{\Gamma} u \, \mathrm{d}\mu$. From here, one just imitate the same procedure used for the interior norms to conclude the validity of the equivalent boundary norms, as desired.
\end{proof}  

We now define the Banach space 
\begin{equation}\label{7.21}
    \mathbb{W}_{0}(\overline{\Omega}):= \left \{ u \in W^{s,p}_{0}(\Omega) \mid u \in \mathbb{B}^{q}_{\eta,0}(\Gamma) \right \},
\end{equation}
endowed with the norm $\|\cdot \|_{_{\mathbb{W}_{0}(\overline{\Omega})}}:=\|\cdot \|_{_{\mathbb{W}_{p,q}(\Omega,\Gamma) }}$. From the definition of $\|\cdot\|_{_{\mathbb{W}_{0}(\overline{\Omega})}}$ and Proposition \ref{equi-norms}, we obtain that the norm 
\begin{equation}\label{norm-equivalent}
    \|u\|_{_{0,\overline{\Omega}}}:= \left [\mathfrak{N}_{_{\Omega}}^{p}(u) \right ]^{1/p}+ \left [\mathfrak{N}_{_{\Gamma}}^{q}(u) \right ]^{1/q}
\end{equation}
defines an equivalent norm for the space $\mathbb{W}_{0}(\overline{\Omega})$.\\
\indent Next, we define the functional 
$\varPsi \, : \, L^{2}(\Omega) \rightarrow [0,\infty]$, by:
\begin{equation}\label{1.76}
\varPsi(u):=\left \{ \begin{array}{ll}
         \frac{1}{p}\mathfrak{N}_{_{\Omega}}^{p}(u) + \frac{1}{q}\mathfrak{N}_{_{\Gamma}}^{q}(u),&  \text{if} \qquad u \in D(\varPsi), 
 \\[1ex]
         +\infty & \text{if} \qquad u \in L^{2}(\Omega) \, \backslash \, D(\varPsi),
    \end{array}\right.
\end{equation}
with effective domain $D(\varPsi):=\mathbb{W}_{p,q}(\Omega,\Gamma) \cap L^{2}(\Omega) $. We have the following simple result (whose proof is standard and thus, it will be omitted). 

\begin{proposition}\label{lemma1}
   Let $p \in \left [\frac{2N}{N+2s}, \infty \right )$, $q \in \left[\frac{2d}{d + 2\eta}, \infty \right)$. Then the functional $\varPsi$ define by \eqref{1.76} is a proper, convex, lower semi-continuous functional on $L^{2}(\Omega)$. Furthermore, if $\partial \varPsi $ is the corresponding subdifferential associated to the functional $\varPsi$, let $f \in L^{2}(\Omega)$ and $u \in \mathbb{W}_{p,q}(\Omega,\Gamma) \cap L^{2}(\Omega)$. Then $f \in \partial \varPsi(u) $ if and only if
\begin{equation}\label{7.0}
    \left \{  \begin{array}{ll}
        (-\Delta)^s_{_{p,\Omega}}u =f, & \text{in} \quad \Omega,  \\ \\
         C_{p,s}\mathcal{N}^{p'(1-s)}_{p}u+\Theta^{\eta}_qu = 0,& \text{on} \quad \Gamma.
    \end{array}  \right.
\end{equation}
\end{proposition} 

\begin{remark}\label{varform}
    Observe that the variational formulation of equation \eqref{7.0} is given by
 \begin{align}\label{7.01}
    \mathcal{E}(u,v):= \,& C_{N,p,s} \int_{\Omega} \int_{\Omega} \frac{|u(x)-u(y)|^{p-2}(u(x)-u(y))(v(x)-v(y))}{|x-y|^{N+sp}}  \, \mathrm{d}x \, \mathrm{d}y \nonumber \\ \nonumber \\
    &+ \int_{\Gamma} \int_{\Gamma} \frac{|u(x)-u(y)|^{q-2}(u(x)-u(y))(v(x)-v(y))}{|x-y|^{d+\eta q}} \, \mathrm{d}\mu_{x} \, \mathrm{d}\mu_{y},
\end{align}
for all $u,v \in \mathbb{W}_{p,q}(\Omega,\Gamma) \cap L^{2}(\Omega)$. Therefore, a function $u \in \mathbb{W}_{p,q}(\Omega,\Gamma) \cap L^{2}(\Omega)$ is a weak solution to problem \eqref{7.0} if  and only if
\begin{equation}
    \mathcal{E}(u,\varphi)=\int_{\Omega}f\varphi \, \mathrm{d}x, \qquad \forall \varphi \in \mathbb{W}_{p,q}(\Omega,\Gamma) \cap L^{2}(\Omega).
\end{equation}
\end{remark}

In the next two results, we establish a relation between the null space of the operator $\mathcal{A}:=\partial \varPsi $ and its range.

\begin{lemma}\label{null}
    Let $\mathscr{N}(\mathcal{A})$ denote the null space of the operator $\mathcal{A}$. Then
    
    \begin{equation*}
        \mathscr{N}(\mathcal{A})=C:=\{c \in \mathbb{R} \} ,
    \end{equation*}
that is, $\mathscr{N}(\mathcal{A})$ consists of all the real constant functions on $\Omega$.
\end{lemma}
\begin{proof}
By Lemma \ref{lemma1}, we have that $ u \in \mathscr{N}(\mathcal{A}) $ if and only if $ u $ is a weak solution of \eqref{7.0} with $ f = 0 $. Moreover, by Remark \ref{varform}, this is equivalent to $ u \in \mathscr{N}(\mathcal{A}) $ if and only if
\begin{equation}\label{7.5}
    \mathcal{E}(u,v)=0, \qquad \forall v \in \mathbb{W}_{p,q}(\Omega,\Gamma) \cap L^{2}(\Omega).
\end{equation}
First, assume that $ u \in \mathscr{N}(\mathcal{A}) $. Taking $ v = u $ in \eqref{7.5}, we obtain  
\begin{equation*}
    \mathcal{E}(u,u)=C_{N,p,s}\int_{\Omega}\int_{\Omega}\frac{|u(x)-u(y)|^{p}}
{|x-y|^{N+sp}}\,\mathrm{d}x \, \mathrm{d}y+\int_{\Gamma}\int_{\Gamma}\frac{|u(x)-u(y)|^{q}}
{|x-y|^{\eta q+d}}\,\mathrm{d}\mu_{x} \, \mathrm{d}\mu_{y}=0.
\end{equation*}
The above equality implies that $ |u(x) - u(y)| = 0 $ a.e. in $ \overline{\Omega} $, which means that $ u(x) = u(y) $ a.e. Hence, $ u(x) = c $ for some constant $ c \in \mathbb{R} $. 
Conversely, if $ u(x) = c $ for some constant $ c \in \mathbb{R} $, then it is clear that $ u $ satisfies \eqref{7.5}, which implies that $ u \in \mathscr{N}(\mathcal{A}) $.
\end{proof}

\begin{lemma}\label{range}
     The range of the operator $\mathcal{A}$ is given by 
\begin{equation*}
    \mathscr{R}(\mathcal{A})= \left \{ f \in L^{2}(\Omega) \ | \ \int_{\Omega} f \, \mathrm{d}x =0 \right \}.
\end{equation*}
\end{lemma}
\begin{proof}
    Let $f \in \mathscr{R}(\mathcal{A}) \subseteq L^{2}(\Omega)$. Then, there exists $u \in D(\varPsi)$ such that $\mathcal{A}(u) = f$. In other words, by Lemma \ref{lemma1} and Remark \ref{varform}, we have that  
\begin{equation}\label{7.20}
        \mathcal{E}(u,v)=\int_{\Omega}fv \, \mathrm{d}x, \qquad v \in \mathbb{W}_{p,q}(\Omega,\Gamma) \cap L^{2}(\Omega).
    \end{equation}
In particular, by setting $v = 1 \in \mathbb{W}_{p,q}(\Omega,\Gamma) \cap L^{2}(\Omega)$ in \eqref{7.20}, we obtain  
\begin{equation*}
        0= \mathcal{E}(u,1)=\int_{\Omega}f \, \mathrm{d}x,
    \end{equation*}
    from which it follows that
 \begin{equation*}
    \mathscr{R}(\mathcal{A})\subseteq \left \{ f \in L^{2}(\Omega) \ | \ \int_{\Omega} f \, \mathrm{d}x =0 \right \}.
\end{equation*}
Conversely, let $f \in L^{2}(\Omega)$ be such that $\int_{\Omega} f \, \mathrm{d}x = 0$, and define the function space
\begin{equation*}
    \mathbb{W}_{0,2}(\overline{\Omega}):=\mathbb{W}_{0}(\overline{\Omega}) \cap L^{2}(\Omega) 
\end{equation*}
endowed with the norm
\begin{equation*}
    \| \cdot \|_{_{\mathbb{W}_{0,2}(\overline{\Omega})}}:=\| \cdot \|_{_{\mathbb{W}_{0}(\overline{\Omega})}}+\|\cdot \|_{_{2,\Omega}},
\end{equation*}
where $\mathbb{W}_{0}(\overline{\Omega})$ is defined by \eqref{7.21}. Clearly, $\mathbb{W}_{0,2}(\overline{\Omega})$ is a closed linear subspace of $\mathbb{W}_{p,q}(\Omega,\Gamma) \cap L^{2}(\Omega)$ and therefore is a reflexive Banach space.
Now, define the functional $\mathcal{F}: \mathbb{W}_{0,2}(\overline{\Omega}) \rightarrow \mathbb{R}$ by:
\begin{equation*}
    \mathcal{F}(u):= \frac{1}{p}\mathfrak{N}_{_{\Omega}}^{p}(u) + \frac{1}{q}\mathfrak{N}_{_{\Gamma}}^{q}(u)-\int_{\Omega}fu \, \mathrm{d}x.
\end{equation*}
By H\"older's inequality and the norm equivalence given in \eqref{norm-equivalent}, it follows that
\begin{equation*}
    \displaystyle\int_{\Omega}fu \, \mathrm{d}x  \leq   \| f\|_{_{2,\Omega}} \|u \|_{_{2,\Omega}}\leq C\|f \|_{_{2,\Omega}} \left (\left [\mathfrak{N}_{_{\Omega}}^{p}(u) \right ]^{1/p}+ \left [\mathfrak{N}_{_{\Gamma}}^{q}(u) \right ]^{1/q} \right ), 
\end{equation*}
and the above inequality implies that
    $$\mathcal{F}(u)\geq  \frac{1}{p}\mathfrak{N}_{_{\Omega}}^{p}(u) + \frac{1}{q}\mathfrak{N}_{_{\Gamma}}^{q}(u)-C\|f \|_{_{2,\Omega}} \left (\left [\mathfrak{N}_{_{\Omega}}^{p}(u) \right ]^{1/p}+ \left [\mathfrak{N}_{_{\Gamma}}^{q}(u) \right ]^{1/q} \right ).$$
This result implies that 
\begin{equation*}
    \frac{\mathcal{F}(u)}{\|u\|_{_{\mathbb{W}_{0,2}(\overline{\Omega})}}}\geq \frac{\frac{1}{p}\mathfrak{N}_{_{\Omega}}^{p}(u) + \frac{1}{q}\mathfrak{N}_{_{\Gamma}}^{q}(u)}{\left [\mathfrak{N}_{_{\Omega}}^{p}(u) \right ]^{1/p}+ \left [\mathfrak{N}_{_{\Gamma}}^{q}(u) \right ]^{1/q}}-C\|f \|_{_{2,\Omega}}\geq \left(\frac{1}{p}\mathfrak{N}_{_{\Omega}}^{p}(u) + \frac{1}{q}\mathfrak{N}_{_{\Gamma}}^{q}(u)\right)^{1-\eta_{p,q}}-C\|f \|_{_{2,\Omega}},
\end{equation*}
where $$\eta_{p,q}:=\min\left\{\frac{1}{p},\frac{1}{q}\right\}\chi_{_{\left\{\frac{1}{p}\mathfrak{N}_{_{\Omega}}^{p}(u) + \frac{1}{q}\mathfrak{N}_{_{\Gamma}}^{q}(u)>1\right\}}}+\max\left\{\frac{1}{p},\frac{1}{q}\right\}\chi_{_{\left\{\frac{1}{p}\mathfrak{N}_{_{\Omega}}^{p}(u) + \frac{1}{q}\mathfrak{N}_{_{\Gamma}}^{q}(u)\leq1\right\}}}<1.$$
From this, we deduce that  
\begin{equation*}
    \lim_{\|u\|_{_{\mathbb{W}_{0,2}(\overline{\Omega})}} \to \infty} \frac{\mathcal{F}(u)}{\|u\|_{_{\mathbb{W}_{0,2}(\overline{\Omega})}}} = +\infty,
\end{equation*}
which shows that the functional $\mathcal{F}$ is coercive. Furthermore, from Lemma \ref{lemma1}, we see that $\mathcal{F}$ is convex and lower-semicontinuous on $ L^{2}(\Omega)$, and thus, it follows from \cite[Theorem 3.3.4]{attouch2014variational} that there exists a function $u^* \in \mathbb{W}_{0,2}(\overline{\Omega})$ that minimizes $\mathcal{F}$ in the sense that $\mathcal{F}(u^*) \leq \mathcal{F}(v)$ for all $v \in \mathbb{W}_{0,2}(\overline{\Omega})$. Henceforth, for every $0 < t \leq 1$ and every $v \in \mathbb{W}_{0,2}(\overline{\Omega})$,
\begin{equation*}
    \mathcal{F}(u^*+tv)-\mathcal{F}(u^*)\geq 0.
\end{equation*}
Hence,
\begin{equation*}
    \lim_{t \rightarrow 0^+}\frac{\mathcal{F}(u^*+tv)-\mathcal{F}(u^*)}{t} \geq 0.
\end{equation*}
Applying the Lebesgue Dominated Convergence Theorem, we obtain  
\begin{align}\label{7.12}
    0 & \leq \lim_{t \rightarrow 0^+}\frac{\mathcal{F}(u^*+tv)-\mathcal{F}(u^*)}{t} \nonumber \\ 
    &= C_{N,p,s} \int_{\Omega} \int_{\Omega} \frac{|u^*(x)-u^*(y)|^{p-2}(u^*(x)-u^*(y))(v(x)-v(y))}{|x-y|^{N+sp}}  \, \mathrm{d}x \, \mathrm{d}y \nonumber \\ 
    &\hspace{15pt}+ \int_{\Gamma} \int_{\Gamma} \frac{|u^*(x)-u^*(y)|^{q-2}(u^*(x)-u^*(y))(v(x)-v(y))}{|x-y|^{d+\eta q}} \, \mathrm{d}\mu_{x} \, \mathrm{d}\mu_{y} - \int_{\Omega} fv \, \mathrm{d}x.
\end{align}
Changing $v$ to $-v$ in  \eqref{7.12} implies that 
\begin{equation}\label{7.13}
    \mathcal{E}(u^*,v)=\int_{\Omega}fv \, \mathrm{d}x, \qquad \forall  v \in \mathbb{W}_{0,2}(\overline{\Omega}).
\end{equation}
Now, let $v \in \mathbb{W}_{p,q}(\Omega,\Gamma) \cap L^{2}(\Omega)$.  
Since the function $v - v_{_{\Omega}}$ integrates to zero over $\Omega$, it follows that $v - v_{_{\Omega}} \in \mathbb{W}_{0,2}(\overline{\Omega})$.  
Thus, testing \eqref{7.13} with $v - v_{_{\Omega}}$ and using the assumption that $\int_{\Omega} f \, \mathrm{d}x = 0$, we conclude that  
\begin{equation}
    \mathcal{E}(u^*,v)=\mathcal{E}(u^*,v-v_{_{\Omega}})=\int_{\Omega}f(v-v_{_{\Omega}}) \, \mathrm{d}x=\int_{\Omega}fv \, \mathrm{d}x,
\end{equation}
for all $v \in \mathbb{W}_{p,q}(\Omega,\Gamma) \cap L^{2}(\Omega)$. This establishes the existence of $u^* \in \mathbb{W}_{p,q}(\Omega,\Gamma) \cap L^{2}(\Omega)$ such that $\mathcal{A}(u^*)=f$. Consequently, $f \in \mathcal{R}(\mathcal{A})$, completing the proof.
\end{proof}

We are now ready to state and prove the main result of this section.

\begin{theorem}\label{Fredholm}
    The operator $\mathcal{A}= \partial \varPsi$ satisfies the following type of quasi-linear Fredholm alternative:
 \begin{equation}
        \mathscr{R}(\mathcal{A})=\mathscr{N}(\mathcal{A})^{\perp}.
    \end{equation}
\end{theorem}
\begin{proof}
    This is a direct consequence of Lemma \ref{null} and Lemma \ref{range}.
\end{proof}
\indent\\

\noindent{\bf Acknowledgements.}\, We would like to thank Dr. Yves Achdou and Dr. Nicoletta Tchou for their help in providing us with suitable images of ramified domains.\\

\noindent{\bf Funding Information.}\,  The second author (the corresponding author) was supported by:
{\it The Puerto Rico Science, Technology and Research Trust}, under agreement number 2022-00014.\\

\noindent{\bf Author Contribution.}\, All authors have accepted responsibility for the entire content of this manuscript and consented to its
submission to the journal, reviewed all the results and approved the final version of the manuscript. Alejandro V\'elez-Santiago (corresponding author) was the author of the idea of the project, which was given to Efren Mesino-Espinosa (first author) as a topic for his Master's thesis. Therefore, the first author worked out most of the calculations and derivations of the manuscript under continuous guidance of the corresponding author. All the statements and calculations were refined further by the corresponding author, who also built the introduction section. Combining these procedures, we were able to complete the current manuscript.\\

\noindent{\bf Conflict of Interest.}\, Authors state no conflict of interest.

\indent\\

\noindent{\bf Disclaimer.}  {\it This content is only the responsibility of the authors and does not necessarily represent the official views of The Puerto Rico Science, Technology and Research Trust}.\\

\end{document}